\documentclass[12pt, twoside, english, headsepline]{article}

\usepackage[margin=30mm]{geometry}

\usepackage{scrlayer-scrpage}  % header and footer for KOMA-Script
\clearscrheadfoot                 % deletes header/footer
\cehead[]{F. Cinque and M. Cintoli}        % equal page, right position (inner) 
\cohead[]{Multidimensional random motions with finite velocities}        % odd   page, left  position (inner) 
\cefoot[\pagemark]{\pagemark}  % even page, center position
\cofoot[\pagemark]{\pagemark}     % odd   page, center position

\usepackage{sectsty}
\sectionfont{\fontsize{14}{13}\selectfont}
\subsectionfont{\fontsize{12}{13}\selectfont}

\usepackage{tikz,graphicx}
\usepackage{amsmath,amsthm,amssymb,amsfonts, mathtools, amsbsy, dsfont}
\usepackage[utf8]{inputenx}
\usepackage{xcolor}
\usepackage{microtype}
\usepackage{hyperref}
\hypersetup{linkcolor=blue}

\newcommand*\dif{\mathop{}\!\mathrm{d}}  % simbolo differenziale
\newcommand*{\doi}[1]{\href{http://dx.doi.org/#1}{DOI: #1}}  % link DOI
\newcommand*{\orcid}[1]{\href{https://orcid.org/#1}{ORCID: #1}}

\allowdisplaybreaks % fa spezzare le formule ``equation'' tra le pagine

\newtheorem{theorem}{Theorem}[section] %[theorem]
\newtheorem{proposition}{Proposition}[section] %[theorem]
\newtheorem{lemma}[theorem]{Lemma}
\theoremstyle{definition} % plain
 
\newtheorem{example}{Example}[section]
\newtheorem{remark}{Remark}[section] %[theorem]
\numberwithin{equation}{section} % equazioni con numero interno alla sezione

\providecommand{\keywords}[1]
{
  \small	
  \textbf{\textit{Keywords: }} #1
}
\providecommand{\MSC}[1]
{
  \small	
  \textit{2020 MSC: } #1   
}

\title{Multidimensional random motions with a natural number of finite velocities}
\author{Fabrizio Cinque$^1$ and Mattia Cintoli\\%$^2$\\
        \small Department of Statistical Sciences, Sapienza University of Rome, Italy \\
        \small $^1$fabrizio.cinque@uniroma1.it, \orcid{0000-0002-9981-149X}\\
}

\begin{document}

\maketitle

\begin{abstract}
We present a detailed analysis of random motions moving in higher spaces with a natural number of velocities. In the case of the so-called minimal random dynamics, under some wide assumptions, we show the joint distribution of the position of the motion (both for the inner part and the boundary of the support) and the number of displacements performed with each velocity. Explicit results for cyclic and complete motions are derived. We establish useful relationships between motions moving in different spaces and we derive the form of the distribution of the movements in arbitrary dimension. Finally, we investigate further properties for stochastic motions governed by non-homogeneous Poisson processes.
\end{abstract} \hspace{10pt}

\keywords{Motions In Higher Space; Telegraph Process; Convexity; Partial Differential Equations; Non-homogeneous Poisson Process}

\MSC{Primary 60K99; 60G50}

% ------  CORPO  ---------------------------------------------------------------------------------------
\section{Introduction}

Since the papers of Goldestein \cite{G1951} and Kac \cite{K1974}, who first studied the connection between random displacements of a particle moving back and forth on the line with stochastic times and hyperbolic partial differential equations, researchers have shown increasing interest in the study of finite-velocity stochastic dynamics. The (initial) analytic approach led to fundamental results such as the explicit derivation of the distribution of the so-called telegraph process \cite{BNO2001, I2001, O1990}, the progenitor of all random motions that later appeared in the literature (also see \cite{C2022a, DcIMZ2013} for further explicit results and Cinque \cite{C2022b} for the description of a reflection principle holding for one-dimensional finite-velocity motions). However, as the number of possible directions increases, the order of the partial differential equation (pde) governing the probability distribution of the absolutely continuous component of the stochastic movement increases as well; in particular, as shown for the planar case by Kolesnik and Turbin \cite{KT1998}, the order of the governing pde coincides with the number of velocities that the motion can undertake. To overcome the weakness of the analytical approach, different ways have been presented to deal with motions in spaces of higher order. One of the first explicit results for multidimensional processes concerned a two-dimensional motion moving with three velocities, see \cite{Dc2002, O2002}, extended to different rules of change of directions in \cite{LO2004}. We also remark the papers of Kolesnik and Orsingher \cite{KO2005}, dealing with a planar motion choosing between the continuum spectrum of the possible directions of the plane, $(\cos \alpha, \sin \alpha), \ \alpha\in [0,2\pi]$, De Gregorio \cite{Dg2010} and Orsingher and De Gregorio \cite{ODg2007}, where are respectively analyzed the corresponding motions on the line and on higher spaces (note that here we only consider motions with a finite number of velocities). Very interesting results concerning motions in arbitrary dimensions were also presented by Samoilenko \cite{S2001} and then furtherly investigated by Lachal \textit{et al.} \cite{LLO2006} and Lachal \cite{L2006}; see \cite{GO2016, K2021, P2012} as well. It is worth recalling that, under some specific assumptions, explicit and fascinating results have been derived; for example in the case of motions moving with orthogonal directions \cite{CO2023, CO2022, OGZ2020, OK1996}. We also reference the papers \cite{DcIM2023, IV2023} for motions driven by geometric counting processes. Along the years, also physicians studied in depth stochastic motions with finite-velocities accomplishing interesting outcomes, see for instance \cite{ML2020, MDMS2020, SBS2020}.

Random evolutions represent a realistic alternative to diffusion processes to suitably model real phenomena in several fields. In geology, to represent the oscillations of the ground \cite{TDcMS2018}, in physics, to describe the random movements of electrons in a conductor or the bacterial dynamics \cite{MADlB2012}, or the movements of particles in gases, and in finance, to model stock prices \cite{KR2013}.

In this paper we present some general results for a wide class of random motions moving with a natural number of finite velocities. After a detailed introduction on the probabilistic description of these stochastic processes, we begin our study focusing on minimal motions, i.e. those one moving with the minimum number of velocities in order to have the state space of the same dimension of the space where they develop. In this case we derive the exact probability in terms of their basic components, generalizing the known results in the current literature. The probabilities concern both the inner part and the boundary of the support of the moving particle. Furthermore, thanks to a one-to-one correspondence between minimal stochastic dynamics, we introduce a \textit{canonical (minimal) motion} to help the analysis and to show explicit results. The provided examples concern different type of motions governed by both Poisson-type processes and geometric counting processes. In Section 3 we derive the distribution of a motion moving with an arbitrary number of velocities by connecting the problem to minimal movements. Finally, in Section 4, we recover the analytic approach to show some characteristics of stochastic dynamics driven by a non-homogeneous Poisson process; in particular, the relationships between the conditional probability of movements in higher order and lower dimensional dynamics.

\subsection{Random motions with a natural number of finite velocities}

Let $\bigl(\Omega, \mathcal{F},\{\mathcal{F}_t\}_{t\ge0}, P\bigr)$ be a filtered probability space and $D \in \mathbb{N}$. In the following we assume that every random object is suitably defined on the above probability space (i.e. if we introduce a stochastic process, this is adapted to the given filtration).
\\

Let $\{W_n\}_{n\in \mathbb{N}_0}$ be a sequence of random variable such that $W_n \ge 0 \ a.s.\ \forall \ n$ and $W_0 = 0\ a.s.$. Let us define $T_n = \sum_{i=0}^n W_i,\ n\in \mathbb{N}_0$, and the corresponding point process $N = \{N(t)\}_{t\ge0}$ such that $N(t) = \max\{n\in\mathbb{N}_0\,:\,\sum_{i=1}^n W_i \le t\}\ \forall\ t$. Unless differently described, we assume $N$ such that $N(t)<\infty \ \forall \ t\ge0\ a.s.$. Set also $V = \{V(t)\}_{t\ge0}$ be a stochastic vector process taking values in a finite state space $\{v_0,\dots, v_M\}\subset \mathbb{R}^D, \ M\in\mathbb{N}$, and such that $P\{V(t+\dif t) \not = V(t)\,|\,N(t,t+\dif t] = 0\} = 0, t\ge0$.
Now, we can introduce the main object of our study, the $D$-dimensional \textit{random motion (with a natural number of finite velocities)} $X = \{X(t)\}_{t\ge0}$ with velocity given by $V$, i.e. moving with the velocities $v_0,\dots, v_M$ and with displacements governed by the random process $N$, 
\begin{equation}\label{definizioneMoto}
X(t) = \int_0^t V(s)\dif s = \sum_{i=0}^{N(t)-1} \bigl(T_{i+1} -T_i\bigr) V(T_i) + \bigl(t-T_{N(t)}\bigr) V(T_{N(t)}), \ \ t\ge0,
\end{equation}
where $V(T_i)$ denotes the random speed after the $i$-th event recorded by $N$, therefore after the potential switch occurring at time $T_i$ (clearly, $T_{i+1}-T_i = W_{i+1}$). The stochastic process $X$ describes the position of a particle moving in a $D$-dimensional (real) space with velocities $v_0, \dots, v_M$ and which can change its velocity only when the process $N$ records a new event.
\\For the sake of brevity we also call $X$ \textit{finite-velocity random motion} (even though this definition suites also for a motion with an infinite number of finite velocities).

\begin{example}[Telegraph process and cyclic motions]\label{esempioCiclico}
If $D=1$, $N$ homogeneous Poisson process with rate $\lambda>0$ and $v_0 = -v_1= c>0$ such that these velocities alternate, i.e. $V(t) = V(0)(-1)^{N(t)}, t\ge0$, then we have the well-known symmetric telegraph process, describing the position of a particle moving back and forth on the line with exponential displacements of average length $c/\lambda$.

In literature, random motions where the velocities change with a deterministic order are usually called \textit{cyclic} motions. If $X$ is a $D$-dimensional motion with $M+1$ velocities, we say that it is cyclic if (without any loss of generality) $P\{V(t+\dif t ) = v_h\,|\,V(t)=v_j,\, N(t,t+\dif t]=1\} = 1$ for $h=j+1$, and $0$ otherwise, $\forall\ j,h$, where $N$ is the point process governing the displacements (and $v_{h+k(M+1)} = v_h\ k\in \mathbb{Z},h=0,\dots,M$). For a complete analysis on this type of motions see \cite{L2006, LLO2006}.
\hfill$\diamond$
\end{example}

\begin{example}[Complete random motions]\label{motoUniforme}
If $P\{V(0)=v_h\}>0$ and $p_{j,h}=P\{V(t+\dif t) = v_h\,|\, V(t)=v_j,\, N(t, t+\dif t]= 1\}  > 0$ for each $j,h = 0,\dots, M$, we call $X$ \textit{complete} random motion. In this case, at each event recorded by the counting process $N$, the particle can switch velocity to any of the available ones (with strictly positive probability). 
%If $p_{i,j} = 1/(M+1),\ \forall\ i,j$, we call $X$ complete uniform random motion.
\hfill$\diamond$
\end{example}

\begin{example}[Random motion with orthogonal velocities]
Put $D=2$. The motion $(X,Y)$ moving in $\mathbb{R}^2$ with the four orthogonal velocities $v_h = \bigl(\cos(h\pi/2), \sin(h\pi/2) \bigr),\ h=0,1,2,3$ such that $P\{V(t+\dif t) =v_h\,|\, V(t) = v_j,\,N(t,t+\dif t] = 1\} =1/2$ if $j = 0,2$ and $h=1,3$ or $j=1,3$ and $h = 0,2$ (i.e. it always switches ``to a different dimension''). $(X,Y)$ is the so-called \textit{standard orthogonal planar random motion}, which, if $N$ is a non-homogeneous Poisson process, can be expressed as a linear function of two independent and equivalent one-dimensional (non-homogeneous) telegraph processes, see \cite{CO2023}. One can imagine also other rules for the changes of velocity, we refer to \cite{CO2023, CO2022, OGZ2020} for further details.\hfill$\diamond$
\end{example}

The support of the random variable $X(t)$ expands as the time increases and it reads
\begin{equation}\label{supportoMoto}
\text{Supp}\bigl(X(t)\bigr) = \text{Conv}(v_0t,\dots, v_Mt), \ \ t\ge0,
\end{equation}
where Conv$(\cdot)$ denotes the convex hull of the input vectors. 
Therefore, the motion $X$ moves in a convex polytope of dimension
\begin{equation}
\text{dim}\Bigl(\text{Conv}(v_0,\dots,v_M)\Bigr) = \text{rank}\Bigl(v_1-v_0\ \cdots\ v_M-v_0\Bigr)
 = \text{rank}\begin{pmatrix}\begin{array}{l}
1^T\\
\mathrm{V}
\end{array}\end{pmatrix} -1,
\end{equation}
where $\mathrm{V} = (v_0\ \cdots\ v_M)$ is the matrix with the velocities as columns and $1^T$ is a row vector of all ones (with suitable dimension). For $H = 0,\dots, M$, if the particle takes all and only the velocities $v_{i_0}, \dots, v_{i_H}$ in the time interval $[0,t]$, then it is located in the set $\overset{\circ}{\text{Conv}}(v_{i_0}t, \dots, v_{i_H}t)$ (where $\overset{\circ}{S}$ denotes the inner part of the set $S\subset\mathbb{R}^D$ and we assume the notation $\overset{\circ}{\text{Conv}}(v) = \{v\},\, v\in \mathbb{R}^D$).
\\

%In the following we denote with $\mathcal{C}_H^S$ the set of combinations of $H$ elements from the set $S\subset \mathbb{R}^D$, with $H<|S|<\infty$. Furthermore, we define $v_0,\dots,v_M$ \textit{$H$-affinely independent}, with natural $H\le D$, if $\exists\ i\in \mathcal{C}_{H}^{\{1,\dots,D\}}$ such that their projections onto $\mathbb{R}^H$ given by the components in $i$, $v_0^{i}, \dots, v_M^{i}$, are affinely independent, i.e. if $\exists \ p_{i}:\mathbb{R}^M \longrightarrow \mathbb{R}^H$ such that $p_i(v_j)=v_j^i, j=0,\dots, M$, are affinely independent. Clearly, the vectors are $D$-affinely independent if and only if they are affinely independent.
%\\It is important to observe that, by assuming $r = \text{rank}(\mathrm{V})$,
%\begin{itemize}
%\item[($i$)] if $r<M$, then the velocities $v_0,\dots,v_M$ (column vectors of $\mathrm{V}$) are affinely dependent,
%\item[($ii$)] if $r\ge M$ and $v_0,\dots,v_M$ are affinely independent, then $\text{rank}\begin{pmatrix}\begin{array}{l}
%1^T\\
%\mathrm{V}
%\end{array}\end{pmatrix} = r+1$ and the vectors are $H$-affinely independent for $r\le H\le D$.
%\end{itemize}

Our analysis involves the relationships between motions moving in spaces of different orders or with state spaces of different dimensions. From (\ref{definizioneMoto}) it is easy to check that, if $A$ is a $R\times D$-dimensional real matrix, then the motion $X' = \{AX(t)\}_{t\ge0}$ is a $R$-dimensional motion governed by $N$ and with velocities $v_0',\dots,v'_M\in \mathbb{R}^R$ such that $v_h' = Av_h,\ \forall\ h$.

In the following we are using the next lemma from affine geometry theory.

\begin{lemma}\label{lemmaProiezione}
Let $v_0,\dots,v_M\in \mathbb{R}^D$ such that $\text{dim}\Bigl(\text{Conv}(v_0,\dots,v_M)\Bigr) = R$. For $k=0,\dots, M$, there exists the set $I^{R,k}$ of the indexes of the first $R$ linearly independent rows of the matrix $\Bigl[v_h-v_k\Bigr]_{\substack{h=0,\dots,M\\h\not=k}}$ and $I^R = I^{R,k} = I^{R,l} \  \forall\ k,l$.
\\Let $e_1,\dots,e_D$ be the vector of the standard basis of $\mathbb{R}^D$, then the orthogonal projection $p_{R}:\mathbb{R}^D\longrightarrow\mathbb{R}^R,$ $p_{R}(x) = \Bigl[e_i\Bigr]^T_{i\in I^{R}}x$, is such that, with $v_h^{R} = p_{R}(v_h)\ \forall \ h$,
\\$\text{dim}\Bigl(\text{Conv}(v^{R}_0,\dots,v^{R}_M)\Bigr) = R$ and 
$\,\forall \ x^{R}\in\text{Conv}(v^{R}_0,\dots,v^{R}_M) \ \ \exists\ !\ x\in \text{Conv}(v_0,\dots,v_M)  \ s.t.\ p_{R}(x) = x^{R}$.
\end{lemma}
See Appendix \ref{dimLemmaProiezione} for the proof. Also note that, if $R=M<D$, i.e. the $R+1$ vectors are affinely independent, then the projected vectors $p_R(v_0),\dots, p_R(v_R)$ are affinely independent as well. Obviously, if $R=D$, $p_R$ is the identity function. 

For our aims, the core of Lemma \ref{lemmaProiezione} is that, for a collection of vectors $v_0,\dots,v_M\in \mathbb{R}^D$ such that $\text{dim}\Bigl(\text{Conv}(v_0,\dots,v_M)\Bigr) = R$, there exists an orthogonal projection, $p_R$, onto a $R$-dimensional space such that to each element of this projection of the convex hull, $x^R\in\text{Conv}(v^R_0,\dots,v^R_M)$, corresponds (one and) only one element of the original convex hull, $x\in \text{Conv}(v_0,\dots,v_M)$.

\section{Minimal random motions}

A random motion in $\mathbb{R}^D$ needs $D+1$ affinely independent velocities in order to have a $D$-dimensional state space. Therefore, we say that the $D$-dimensional stochastic motion $X$, defined as in (\ref{definizioneMoto}), is \textit{minimal} if it moves with $D+1$ affinely independent velocities $v_0,\dots,v_D\in \mathbb{R}^D$.

The support of the position $X(t), t\ge0$, of a minimal random motion is given in (\ref{supportoMoto}) and it can be decomposed as follows 
\begin{equation}\label{supportoMotoMinimale}
\text{Supp}\bigl(X(t)\bigr) =  \bigcup_{H=0}^D\ \bigcup_{i\in \mathcal{C}_{H+1}^{\{0,\dots,D\}}}\overset{\circ}{\text{Conv}}(v_{i_0}t, \dots, v_{i_H}t),
\end{equation}
where $\mathcal{C}_{k}^{S}$ denotes the combinations of $k$ elements from the set $S$, with $0\le k\le |S|<\infty$. Since $X$ is a minimal motion, it lies on each convex hull appearing in (\ref{supportoMotoMinimale}) if and only if it moves with all and only the corresponding velocities in the time interval $[0,t]$.
\\

Let us denote with $T_{(h)} = \{T_{(h)}(t)\}_{t\ge0}$ the stochastic process describing, for each $t\ge0$, the random time that the process $X$ spends moving with velocity $v_h$ in the time interval $[0,t]$, with $ h = 0,\dots,D$. In formula, $T_{(h)}(t)  = \int_0^t \mathds{1}(V(s) = v_h) \dif s, \ \forall\ t,h$. Furthermore, we denote with $T_{(\cdot)} = (T_{(0)},\dots,T_{(D)})$ the vector process describing the times spent by the motion moving with each velocity, and with $T_{(k^-)} = (T_{(0)},\dots,T_{(k-1)},T_{(k+1)},\dots,T_{(D)})$ the vector process describing the time that $X$ spends with each velocity except for the $k$-th one, $k=0,\dots, D$. In the next proposition we express $X$ as an affine function of $T_{(k^-)}$.

\begin{proposition}\label{proposizioneXfunzioneT}
Let $X = \{X(t)\}_ {t\ge0}$ be a finite-velocity random motion in $\mathbb{R}^D$ moving with velocities $v_0,\dots,v_D$. For $k = 0,\dots,D$,
\begin{equation}\label{XfunzioneT}
%la t è blu
X(t) = g_k\Bigl(t,T_{(k^-)}(t)\Bigr) = v_kt+\Biggl[v_h - v_k\Biggr]_{\substack{h=0,\dots, D\\h \not = k}} T_{(k^-)}(t),\ \ t\ge0,
\end{equation}
where $\Bigl[v_h - v_k\Bigr]_{\substack{h=0,\dots, D\\h \not = k}} $ denotes the matrix with columns $v_h-v_k, h\not=k$. Furthermore, for fixed $t\ge0$, $g_k$ is bijective $\forall \ k$ if and only if $v_0,\dots,v_D$ are affinely independent (i.e. if and only if $X$ is minimal).
\end{proposition}
Hereafter we are omitting the direct dependence of $g_k$ (and the similar functions) from the time variable $t$ since we are always working with fixed $t\ge0$; thus, we are more briefly writing, for instance, $g_k\bigl(X(t)\bigr)$.
\begin{proof}
Fix $t\ge0$. By definition $\sum_{h=0}^D T_{(h)}(t) = t\ a.s.$ and $X(t) = \sum_{h=0}^D v_h\, T_{(h)}(t)$, therefore, for each $k=0,\dots,D$, $X(t) = \sum_{h\not = k} (v_h-v_k) T_{(h)}(t) + v_k t$ that in matrix form is (\ref{XfunzioneT}).
\\The matrix $\Bigl[v_h - v_k\Bigr]_{h\not = k} $ is invertible $\forall \ k$ if and only if all the differences $v_h-v_k,h\not = k$ are linearly independent $\forall \ k$ and thus if and only if the velocities $v_0,\dots,v_D$ are affinely independent.
\end{proof}

\begin{remark}\label{remarktXfunzioneT}
Another useful representation of a finite-velocity random motion $X$ is, with $t\ge0$,
\begin{equation}\label{tXfunzioneT}
\begin{pmatrix}\begin{array}{l}t\\ X(t)\end{array}\end{pmatrix}  =
g\Bigl(T_{(\cdot)}(t)\Bigr) = \begin{pmatrix}\begin{array}{l}1^T\\ \mathrm{V} \end{array}\end{pmatrix}  T_{(\cdot)}(t)
\end{equation}
and $g$ is bijective if and only if $X$ is minimal (indeed, $\begin{pmatrix}\begin{array}{l}1^T\\ \mathrm{V} \end{array}\end{pmatrix} $ is invertible if and only if the velocities, columns of $\mathrm{V}$, are affinely independent). In this case we write $T_{(\cdot)}(t) = g^{-1}\bigl(t, X(t)\bigr) = \Bigl(g_{\cdot, h}^{-1} \bigl(t, X(t)\bigr) \Bigr)_{h=0,\dots,D} = \begin{pmatrix}\begin{array}{l}1^T\\ \mathrm{V} \end{array}\end{pmatrix}^{-1} \begin{pmatrix}\begin{array}{c}t\\ X(t)\end{array}\end{pmatrix}$. Now, for $k=0,\dots,D$, we write the inverse of (\ref{XfunzioneT}) as
\begin{equation}\label{inversatXfunzioneT}
T_{(k^-)}(t)= g^{-1}_k\bigl(X(t)\bigr) = \Bigl(g_{k, h}^{-1} \bigl(X(t)\bigr) \Bigr)_{h\not =k} =  \Bigl(g_{\cdot, h}^{-1} \bigl(t, X(t)\bigr) \Bigr)_{h\not=k} = \Bigl[v_h-v_k\Bigr]_{h\not = k}^{-1} \bigl(X(t)-v_kt\bigr).
\end{equation}
Notation (\ref{inversatXfunzioneT}) is going to be useful below. Clearly, $ t - \sum_{h\not=k} g_{\cdot, h}^{-1} \bigl(t,X(t)\bigr) = t - \sum_{h\not=k} T_{(h)}(t) = T_{(k)}(t)= g_{\cdot, k}^{-1} \bigl(t,X(t)\bigr)$.
\hfill$\diamond$
\end{remark}

\begin{theorem}\label{teoremaX'funzioneX}
Let $X = \{X(t)\}_ {t\ge0}$ be a minimal random motion moving with velocities $v_0,\dots,v_D\in\mathbb{R}^D$ and whose displacements are governed by a point process $N$. Let $X' = \{X'(t)\}_{t\ge0}$ be a random motion with velocities $v_0',\dots,v_D'$ whose displacements are governed by a point process $N' \stackrel{d}{=}N$ and with the same rule of changes of velocity of $X$. Then, for $t\ge0$,
\begin{equation}
X'(t) \stackrel{d}{=} f\bigl( X(t) \bigr) = \mathrm{V}' \begin{pmatrix}\begin{array}{l}1^T\\ \mathrm{V} \end{array}\end{pmatrix}^{-1}\begin{pmatrix}\begin{array}{c}t\\ X(t) \end{array}\end{pmatrix},
\end{equation}
where $\mathrm{V}' = (v_0',\dots, v_D')$. Furthermore, $f$ is bijective if and only if $X'$ is minimal.
\end{theorem}

\begin{proof}
Fix $t\ge0$. Since the changes of velocity of $X$ and $X'$ have the same rule and $N(t)\stackrel{d}{=}N'(t)\ \forall\ t$, then $T_{(h)}(t)\stackrel{d}{=}T'_{(h)}(t)\ \forall\ h,t$. For $k=0,\dots,D$, $X'(t) = \mathrm{V}'T'_{(\cdot)}(t)\stackrel{d}{=} \mathrm{V}' \begin{pmatrix}\begin{array}{l}1^T\\ \mathrm{V} \end{array}\end{pmatrix}^{-1}\begin{pmatrix}\begin{array}{c}t\\ X(t) \end{array}\end{pmatrix}$. 
Now, for Proposition \ref{proposizioneXfunzioneT}, $\forall\ k$, $X(t)$ is in bijective correspondence with $T_{(k^-)}(t)$ and then with $T'_{(k^-)}(t)$ (in distribution). Therefore $X(t)$ is in bijective correspondence with $X'(t)$ if and only if $X'(t)$ is in bijective correspondence with $T'_{(k^-)}(t)$, so if $X'$ minimal.
\end{proof}

\begin{remark}[Canonical motion]\label{remarkDefMotoCanonico}
Theorem \ref{teoremaX'funzioneX} states that all minimal random motions in $\mathbb{R}^D$, with displacements and changes of directions governed by the same probabilistic rules, are in bijective correspondence (in distribution). Therefore, it is useful to introduce a minimal motion $X =\{X(t)\}_{t\ge0}$ moving with the \textit{canonical} velocities of $\mathbb{R}^D$, $e_0 = 0,e_1,\dots,e_D$, where $e_h$ is the $h$-th vector of the standard basis of $\mathbb{R}^D$. At time $t\ge0$, the support of the position $X(t)$ is given by the convex set $\{x\in \mathbb{R}^D\,:\,x\ge0,\, \sum_{i=1}^D x_i\le t\}$. Put $t\ge0$ and $\mathrm{E} = (e_0 \ \cdots\ e_D) = (0 \ I_D)$, the matrix having the canonical velocities as columns. In view of Remark \ref{remarktXfunzioneT}, the canonical motion can be expressed as $\begin{pmatrix}\begin{array}{c}t\\ X(t) \end{array}\end{pmatrix} = \begin{pmatrix}\begin{array}{l}1^T\\ \mathrm{E} \end{array}\end{pmatrix} T_{(\cdot)}(t)$ and
\begin{equation}\label{inversatXfunzioneTcanonico}
T_{(\cdot)}(t) = g^{-1}\bigl(t, X(t)\bigr) = \begin{pmatrix}\begin{array}{l l}1 & -1^T\\ 0 & I_D \end{array}\end{pmatrix} \begin{pmatrix}\begin{array}{c}t\\ X(t) \end{array}\end{pmatrix} =  \begin{pmatrix}\begin{array}{c}t-\sum_{j=1}^D X_j(t)\\ X_1(t)\\ \cdot \\ X_D(t) \end{array}\end{pmatrix}.
\end{equation}
Keeping in mind notation (\ref{inversatXfunzioneT}), the inverse functions $g_k^{-1},\ k=0,\dots,D$, are given by (\ref{inversatXfunzioneTcanonico}) excluding the $(k+1)$-th term (which concerns the time $T_{(k)}(t)$).

Finally, if $Y$ is a random motion with affinely independent velocities $v_0,\dots,v_D$, under the hypothesis of Theorem \ref{teoremaX'funzioneX}, we can write
\begin{equation}\label{X'funzioneXcanonico}
Y(t) \stackrel{d}{=} v_0 t + \Bigl[v_h - v_0\Bigr]_{h=1,\dots,D}X(t),\ \ \ t\ge0.
\end{equation}
\end{remark}

\subsection{Probability law of the position of minimal motions}\label{sezioneProbabilitaMotiMinimali}

In order to study the probability law of the position $X(t),\, t\ge0$, of a minimal random motion in $\mathbb{R}^D$ we need to make some hypothesis on the probabilistic mechanisms of the process, i.e. on the velocity process $V$ and the associated point process $N$, which governs the displacements. %Hereafter we assume the following hypothesis (unless differently specified).

(H1) The changes of velocity, which can only occur when the point process $N$ records an event, depend only from the previously selected velocities (but not from the moment of the switches nor the time spent with each velocity). Therefore, we assume that, $t\ge0$,
\begin{align}
P&\{V(t+\dif t) = v_h\,|\, N(t, t+\dif t] = 1,\, \mathcal{F}_t\}\nonumber\\
& = P\{V(t+\dif t) = v_h\,|\, N(t, t+\dif t] = 1,\, V(T_j), j=0,\dots, N(t)\}, \label{ipotesiH1}
\end{align}
where $T_0=0\ a.s.$ and $T_1,\dots,T_{N(t)}$ are the arrival times of the process $N$ (see (\ref{definizioneMoto})). Note that if $N$ is Markovian and the conditional event in (\ref{ipotesiH1}) can be reduced to $\{ N(t, t+\dif t] = 1,\, V\bigl(T_{N(t)}\bigr)\}$, then both $V$ and $(X, V)$ are Markovian (see for instance \cite{D1984}).
\\

Now, we define, for $t\ge0,\ h=0,\dots, D$, the processes $N_h(t) = \big|\{0\le j\le N(t)\,:\, V(T_{j}) = v_h\} \big|$ counting the number of displacements with velocity $v_h$ in the time interval $[0,t]$. Clearly $\sum_{h=0}^D N_h(t) = N(t) +1\ a.s.$ (because the displacements are one more than the switches since the initial movement). Let us also define the random vector $C_{N(t)+1} = \bigl(C_{N(t)+1, 0}, \dots, C_{N(t)+1,D}\bigr)\in \mathbb{N}_0^{D+1}$ which provides the allocation of the selected velocities in the $N(t)+1$ displacements.
\\For $t\ge0$ and $n_0,\dots, n_D$, we have the following relationship $\Big\{ \bigcap_{h=0}^D N_h(t) = n_h\Big\} \iff \{ N(t) = n_0+\dots+n_D -1,\, C_{n_0+\dots+n_D} = (n_0, \dots, n_D)\}$.

\begin{example}[Complete random motions]\label{remark2MotoCompleto}
Consider the motion in Example \ref{motoUniforme} with $p_{j,h} = p_h = P\{V(0) = v_h\}>0,\ j,h = 0,\dots, D$, so that the probability of selecting a velocity does not depend on the current velocity. Then, $\forall \ t, \ C_{N(t)+1}\sim Multinomial \bigl(N(t)+1, p = (p_0, \dots, p_D)\bigr)$.
\hfill$\diamond$
\end{example}
 
(H2) The time of the displacements along different velocities are independent, i.e. the waiting times $\{W_n\}_{n\in \mathbb{N}_0}$  (see (\ref{definizioneMoto})) are independent if they concern different velocities. For each $h= 0,\dots,D$, let $\big\{W_n^{(h)}\big\}_{n\in \mathbb{N}}$ be the sequence of the times related to the displacements with velocity $v_h$. In details, $W_n^{(h)}$ denotes the time of the $n$-th movement with velocity $v_h$ and $W_n^{(h)}, W_m^{(k)}$ are independent if $h\not=k, \forall\ m,n$. Let also $N_{(h)} =\{N_{(h)}(s) \}_{s\ge0}$ be the associated point process, i.e. such that $N_{(h)}(s) = \max\{n\,:\, \sum_{i=0}^n W_i^{(h)}\le s\},\ \forall \ s$. Then $N_{(0)},\dots, N_{(D)}$ are independent counting processes.

From hypothesis (H1) we have that the random times $W_n^{(h)}$ are independent of the allocation of the velocities among the steps, i.e. for $A\subset \mathbb{R}$,
\begin{equation}\label{fintaIpotesi3}
P\{W_m^{(h)}\in A,\, V(T_n) = v_h,\, C_{n,h} = m\} = P\{W_m^{(h)}\in A\} P \{V(T_n) = v_h,\, C_{n,h} = m\}
\end{equation}
for each $m\le n\in \mathbb{N},\, h = 0,\dots, D$. In words, the first member of (\ref{fintaIpotesi3}) concerns the $m$-th displacement with speed $v_h$ to be in $A$, also requiring that this is the $n$-th movement of the motion.
\\

Below we use the following notation: for any suitable function $g$ and any suitable absolutely continuous random variable $X$ with probability density $f_X$, we write $P\{X\in \dif\, g(x)\} = f_X\bigl(g(x)\bigr) |J_g(x)|\dif x$, where $J_g$ is the Jacobian matrix of $g$.

\begin{theorem}\label{teoremaLeggeIntersezioniMinimale}
Let $X$ be a minimal finite-velocity random motion in $\mathbb{R}^D$ satisfying (H1)-(H2). For $t\ge0,\,x\in \overset{\circ}{\text{Supp}}\bigl(X(t)\bigr),\ n_0,\dots,n_D\in \mathbb{N}$ and $k=0,\dots, D$,
\begin{align}
& P\bigg\{X(t)\in \dif x,\,\bigcap_{h=0}^D\{N_h(t) = n_h\} ,\, V(t) = v_k\bigg\}\label{leggeIntersezioniMinimale}\\
& = \prod_{\substack{h=0\\h\not=k}}^D f_{\sum_{j=1}^{n_h} W_j^{(h)}}\Bigl(g_{k,h}^{-1}(x)\Bigr) \dif x \,\Big|\bigl[v_h-v_k\bigr]_{h\not =k}\Big|^{-1} \, P\Bigg\{N_{(k)}\biggl(t-\sum_{\substack{h=0\\h\not = k}}^D g_{k,h}^{-1}(x)\biggr) = n_k-1\Bigg\}\nonumber\\
&\ \ \ \times P\big\{C_{n_0+\dots+n_D} = (n_0,\dots,n_D), V(t) = v_k\big\},\nonumber
\end{align}
where $g_k^{-1}$ is given in (\ref{inversatXfunzioneT}).
\end{theorem}

Theorem \ref{teoremaLeggeIntersezioniMinimale} provides a general formula for the distribution of the position of the minimal motion at time $t$ joint with the number of displacements for each direction in the interval $[0,t]$ and the current direction (meaning at time $t$).

\begin{proof}
Fix $k = 0,\dots,D$ and $t\ge0$.
\begin{align}
&P\bigg\{X(t)\in \dif x,\,\bigcap_{h=0}^D\{N_h(t) = n_h\} ,\, V(t) = v_k\bigg\}\nonumber\\
& = P\bigg\{X(t)\in \dif x,\,\bigcap_{h=0}^D\{N_h(t) = n_h\} ,\, \sum_{h=0}^D T_{(h)}(t) = t,\, V(t) = v_k\bigg\}\nonumber \\
& = P\bigg\{X(t)\in \dif x,\,\bigcap_{\substack{h=0\\h\not = k}}^D\Big\{T_{(h)}(t) = \sum_{j=1}^{n_h} W_j^{(h)}\Big\} ,\, N_{(k)}\bigl(T_{(k)}(t)\bigr) = n_k-1,\label{passaggio1teorema}\\
& \ \ \ \ \ \ \ \ C_{n_0+\dots,n_D} =(n_0,\dots,n_D),\, V(t) = v_k\bigg\} \nonumber\\
& =P\bigg\{T_{(k^-)}(t)\in \dif\, g_k^{-1}(x),\,\bigcap_{h\not = k}\Big\{T_{(h)}(t) = \sum_{j=1}^{n_h} W_j^{(h)}\Big\} ,\, N_{(k)}\bigl(T_{(k)}(t)\bigr) = n_k-1,\label{passaggio2teorema}\\
& \ \ \ \ \ \ \ \ C_{n_0+\dots,n_D} =(n_0,\dots,n_D),\, V(t) = v_k\bigg\} \nonumber\\
& =P\bigg\{ \Biggl(\sum_{j=1}^{n_h} W_j^{(h)}\Biggr)_{h\not = k}\in \dif\, g_{k}^{-1}(x) ,\, N_{(k)}\biggl(t-\sum_{h\not = k} g_{k,h}^{-1}(x)\biggr) = n_k-1\bigg\}\label{passaggio3teorema}\\
& \ \ \ \times P\big\{C_{n_0+\dots,n_D} =(n_0,\dots,n_D),\, V(t) = v_k\big\} \nonumber
\end{align}
Step (\ref{passaggio1teorema}) follows by considering that, in the time interval $[0,t]$, the motion performs $n_h$ steps with velocity $v_h,\ \forall\ h$, and it has $V(t) = v_k$  if and only if the total amount of time spent with $v_h$ is given by the sum of $n_h$ waiting times $W_j^{(h)}$, for $h\not = k$ and if the point process $N_{(k)}$ is waiting for the $n_k$-th event at the time $T_{(k)}(t)$ (because $V(t) = v_k$, so the motion completed $n_k-1$ displacements with velocity $v_k$ and it is performing the $n_k$-th). Finally, the event $C_{n_0+\dots,n_D} = (n_0,\dots, n_D)$ pertains the randomness in the allocation of the velocities.
\\Step (\ref{passaggio2teorema}) and (\ref{passaggio3teorema}) respectively follow by considering equation (\ref{XfunzioneT}) and the independence of the waiting times $W_{j}^{(h)}$ from the allocation of the displacements, $\forall\ j,h$, see (\ref{fintaIpotesi3}). 
\\Note that (\ref{passaggio3teorema}) holds for a random motion where the hypothesis (H2) (concerning the independence of the displacements with different velocities) is not assumed. By taking into account (H2) and using (\ref{inversatXfunzioneT}),  (\ref{passaggio3teorema}) coincides with (\ref{leggeIntersezioniMinimale}).
\end{proof}

We point out that if $x\longrightarrow \bar{x}\in\partial \text{Supp}\bigl(X(t)\bigr)$, then for at least one $l=0,\dots, D$, $T_{(l)}(t)\longrightarrow 0 $. Therefore, in (\ref{leggeIntersezioniMinimale}) either $g_{k,h}^{-1}(x) = T_{(h)}(t) \longrightarrow 0 $ for $h\not=k$ or $t-\sum_{h\not=k} g_{k,h}^{-1}(x) = T_{(k)}(t)\longrightarrow 0 $. In light of this observation, for $x$ that tends to the boundary of the support, probability (\ref{leggeIntersezioniMinimale}) goes to $0$ if the density function related to the time $T_{(l)}(t)$ (representing the time which converges to $0$) tends to $0$. See Examples \ref{remarkLeggeMotoCiclico} and \ref{esempioLeggeMotoCanonicoCompleto} for more details.

\begin{remark}[Canonical motion]\label{remarkLeggeMotoCanonico}
If $X$ is a canonical minimal random motion in $\mathbb{R}^D$ (see Remark \ref{remarkDefMotoCanonico}) then, with (\ref{inversatXfunzioneTcanonico}) at hand, we immediately have the corresponding probability (\ref{leggeIntersezioniMinimale}) by considering $g_{\cdot,h}^{-1}(x) = \begin{cases}\begin{array}{l l} t-\sum_{i=1}^Dx_i, &\text{if }h =0,\\x_h,& \text{if }h\not =0, \end{array}\end{cases}$  and the Jacobian determinant is equal to $1$.
\hfill$\diamond$
\end{remark}

\begin{example}[Cyclic motions]\label{remarkLeggeMotoCiclico}
Let $X$ be a cyclic (see Example \ref{esempioCiclico}) minimal motion with velocities $v_0,\dots,v_D\in \mathbb{R}^D$ and let $v_{h+k(D+1)} = v_h, \ h=0,\dots,D,\, k\in \mathbb{Z}$. Let also $N$ be the point process governing the displacements of $X$, then for fixed $t\ge0$, the knowledge of $N(t)$ and $V(t)$ is sufficient to determine $N_h(t)\ \forall\ h$. Let also $P\{V(0) = v_h\} = p_h>0$ and $p_{h+k(D+1)} = p_h,\ \forall \ h$ and $k\in \mathbb{Z}$.
\\Let $n\in \mathbb{N}$ and $k=0,\dots, D$. With $j=1\dots,D$, if the motion performs $n(D+1)+j$ displacements in $[0,t]$, i.e. $N(t) = n(D+1)+j-1$, and $V(t) = v_k$, then $n+1$ displacements occur for each of the velocities $v_{k-j+1}, \dots, v_{k}$ (the $(n+1)$-th displacement with velocity $v_k$ is not complete) and $n$ displacements occur for each of the other velocities $v_{k+1},\dots, v_{k+1+D-j}$. On the other hand, if $j=0$ each velocity is taken $n$ times. Hence, for $x\in \overset{\circ}{\text{Supp}}\bigl(X(t)\bigr)$,
\begin{align}
P&\{X(t)\in \dif x\} = \sum_{j=0}^D \sum_{n=1}^\infty \sum_{k = 0}^D P\{X(t)\in \dif x,\, N(t) =n(D+1)+j-1,\, V(t) = v_k\}\nonumber\\
& = \sum_{j=0}^D \sum_{k = 0}^D \sum_{n=1}^\infty P\bigg\{X(t)\in \dif x,\,\bigcap_{h=k-j+1}^{k}\{N_h(t) = n+1\},\,\bigcap_{h=k+1}^{k+1+D-j}\{N_h(t) = n\} ,\, V(t) = v_k\bigg\},\label{leggeCiclicoMinimale}\
\end{align}
where the probability appearing in (\ref{leggeCiclicoMinimale}) can be derived from (\ref{leggeIntersezioniMinimale}) (note that, for $j=0$, $P\{C_{n(D+1)} = (n,\dots, n),\, V(t) = v_k\} = P\{V(0) = v_{k+1}\}= p_{k+1}$ and, for $j=1,\dots,D$,  with  $n_h = n+1$ if $ h = k-j+1,\dots,k$ and $n_h=n$ if $h = k+1,\dots,k+1+D-j$, $P\{C_{n(D+1)-j} = (n_0,\dots, n_{D}),\, V(t) = v_k\}= P\{V(0) = v_{k-j+1}\}=  p_{k-j+1}$). 
\\
Now, we assume $X$ being a cyclic canonical motion and we derive the probabilities appearing in (\ref{leggeCiclicoMinimale}). For $t\ge0$, in light of Theorem \ref{teoremaLeggeIntersezioniMinimale} and Remark \ref{remarkLeggeMotoCanonico}, by setting $x_0 = t-\sum_{i=1}^D x_i$, we readily arrive at the following distributions, for $x \in \overset{\circ}{\text{Supp}}\bigl(X(t)\bigr) = \{x\in \mathbb{R}^D\,:\,x>0,\, \sum_{i=1}^D x_i< t\}$ and $k=0,\dots, D$. For $j=1,\dots,D,$\begin{align}
P&\bigg\{X(t)\in \dif x,\,\bigcup_{n=1}^\infty N(t) =n(D+1)+j-1,\,V(t) = e_k \bigg\} /\dif x\nonumber\\
& =  \sum_{n=1}^\infty P\{V(0) = v_{k-j+1}\} \nonumber\\
& \ \ \ \times  \Biggl(\prod_{h=k-j+1}^{k-1} f_{\sum_{i=1}^{n+1}W_i^{(h)}}(x_h)\Biggr)\, P\{N_{(k)}(x_k) = n\} \, \Biggl(\prod_{h=k+1}^{k+1+D-j}f_{\sum_{i=1}^{n}W_i^{(h)}}(x_h) \Biggr)\label{formulaCiclicoMinimaleGenerale}
\end{align}
and, for $j=0$,
\begin{align}
P\bigg\{&X(t)\in \dif x,\,\bigcup_{n=1}^\infty N(t) =n(D+1)-1,\,V(t) = e_k\bigg\} /\dif x\nonumber\\
& =  \sum_{n=1}^\infty  P\{V(0) = v_{k+1}\} \Biggl(\prod_{\substack{h=0\\h\not=k}}^{D} f_{\sum_{i=1}^{n}W_i^{(h)}}(x_h)\Biggr)\, P\{N_{(k)}(x_k) = n-1\}.\label{formulaCiclicoMinimaleGeneraleCaso0}
\end{align}

We point out that thanks to relationship (\ref{X'funzioneXcanonico}), from probabilities (\ref{formulaCiclicoMinimaleGenerale}) and (\ref{formulaCiclicoMinimaleGeneraleCaso0}) we immediately obtain the distribution of the position of any $D$-dimensional cyclic minimal random motion, $Y$, moving with velocities $v_0,\dots,v_D$ and governed by a Poisson-type process $N_Y\stackrel{d}{=}N$.

We now present the explicit results for two different types of point processes for $N$.
\\
\\
$a)$ \textit{Homogeneous Poisson-type process.} Assume $N$ a Poisson-type process such that $W_i^{(h)}\sim Exp(\lambda_h), \ i\in \mathbb{N}, h=0,\dots,D$. Then formula (\ref{formulaCiclicoMinimaleGenerale}) turns into
\begin{align}
P&\bigg\{X(t)\in \dif x,\,\bigcup_{n=1}^\infty N(t) =n(D+1)+j-1,\,V(t) = e_k\bigg\} /\dif x\nonumber\\
& =\sum_{n=1}^\infty  p_{k-j+1} \Biggl(\prod_{h=k-j+1}^{k-1} \frac{\lambda_h^{n+1}\,x_h^{n}\,e^{-\lambda_h x_h}}{n!}\Biggr)\, \frac{e^{-\lambda_k x_k}(\lambda_k x_k)^{n}}{n!} \, \Biggl(\prod_{h=k+1}^{k+1+D-j}\frac{\lambda_h^{n}\,x_h^{n-1}\,e^{-\lambda_h x_h}}{(n-1)!} \Biggr) \nonumber\\
& =e^{-\sum_{h=0}^D\lambda_h x_h} \Biggl(\prod_{h=0}^{D} \lambda_h \Biggr)  p_{k-j+1} x_k\Biggl(\prod_{h=k-j+1}^{k-1} \lambda_h x_h \Biggr)\, \tilde{I}_{j, D+1}\Biggl((D+1)\sqrt[D+1]{\prod_{h=0}^{D}\lambda_hx_h}\Biggr),\label{leggeCicloCiclicoMinimale}
\end{align}
where $\tilde{I}_{\alpha,\nu}(z) = \sum_{n=0}^\infty \Bigl(\frac{z}{\nu}\Bigr)^{n\nu} \frac{1}{n!^{\nu-\alpha}(n+1)!^{\alpha}}$, with $0\le\alpha\le\nu,\ z\in \mathbb{C}$, is a Bessel-type function. Similarly, formula (\ref{formulaCiclicoMinimaleGeneraleCaso0}) reads
\begin{align}
P\bigg\{&X(t)\in \dif x,\,\bigcup_{n=1}^\infty N(t) =n(D+1)-1,\,V(t) = e_k\bigg\} /\dif x\nonumber\\
& =e^{-\sum_{h=0}^D\lambda_h x_h} p_{k+1} \Biggl(\prod_{\substack{h=0\\h\not = k}}^{D} \lambda_h \Biggr)\, \tilde{I}_{0, D+1}\Biggl((D+1)\sqrt[D+1]{\prod_{h=0}^{D}\lambda_hx_h}\Biggr).\label{leggeCicloCiclicoMinimaleCaso0}
\end{align}

Note that, if $x\longrightarrow \bar{x}\in \partial \text{Supp}\bigl(X(t)\bigr)$, then  $\exists \ l=0,\dots, D$ such that the total time spent with velocity $l$ goes to $0$, meaning that $T_{(l)}(t) = x_l\longrightarrow0$. With this at hand, we observe that probability (\ref{leggeCicloCiclicoMinimaleCaso0}) reduces to $e^{-\sum_{h\not\in I_0}\lambda_h x_h} p_{k+1} \Bigl(\prod_{h\not = k} \lambda_h \Bigr)$, where $I_0\subset \{0,\dots,D\}$ denotes the indexes of the times going to $0$. Hence, $\forall\ k$, distribution (\ref{leggeCicloCiclicoMinimaleCaso0}) never converges to $0$ for $x$ tending to the boundary of the support. Intuitively, this follows because the probability concerns the event where every velocity is taken exactly $n$ times, with $n\ge1$, and therefore it considers also the case $n=1$, where the random times have exponential density function which is right continuous and strictly positive in $0$.
\\
On the other hand, (\ref{leggeCicloCiclicoMinimale}) can converge to $0$. %(it is easy to see that it can also converge to a positive quantity, for instance if only $x_{D-j+1}\longrightarrow0$). 
In fact, (for fixed $j$) if $D+1-j$ times $T_{(h)}(t) = x_h$ tends to $0$, then for each $k$, at least one of these times appears in $x_k\Bigl(\prod_{h=k-j+1}^{k-1} \lambda_h x_h \Bigr)$ leading it to $0$. This follows because the event in the probability does not include the case where all the velocity whose time converges to $0$ are taken just once.
\\
\\
$b)$ \textit{Geometric counting process.} Assume  that $N_{(h)}, h=0,\dots,D$, are independent geometric counting processes with parameter $\lambda_h>0$; therefore the waiting times $W_i^{(h)}, W_j^{(k)}$ are independent $\forall\ h\not=k$ and they are dependent for $h=k$ and $i\not=j$. In particular, if $M$ is a geometric counting process with parameter $\lambda>0$, then
\begin{equation}\label{leggeProcessoGeometrico}
P\{M(s+t) - M(s) =n\} =\frac{1}{1+\lambda t}\Biggl(\frac{\lambda t}{1+\lambda t}\Biggr)^n ,\ \ \ \ s,t\ge0,\ n\in\mathbb{N}_0,
\end{equation}
and its arrival times have a modified Pareto (Type I) distribution, that is
\begin{equation}\label{intertempiProcessoGeometrico}
P\{T_n \in \dif t\} =\frac{n\lambda}{(1+\lambda t)^2}\Biggl(\frac{\lambda t}{1+\lambda t}\Biggr)^{n-1}\dif t ,\ \ \ \ t\ge0,\ n\in\mathbb{N}.
\end{equation}
We refer to \cite{DcIM2023, IV2023} for further details about geometric counting processes for random motions and to \cite{G1997} for a complete overview on mixed Poisson processes.
\\Now, in light of (\ref{leggeProcessoGeometrico}) and (\ref{intertempiProcessoGeometrico}), formula (\ref{formulaCiclicoMinimaleGenerale}) turns into
\begin{align}
&P\bigg\{X(t)\in \dif x,\,\bigcup_{n=1}^\infty N(t) =n(D+1)+j-1,\,V(t) = e_k\bigg\} /\dif x\nonumber\\
& = \frac{p_{k-j+1}}{1+\lambda_k x_k}\Biggl(\prod_{h=k-j+1}^{k} \frac{\lambda_hx_h}{1+\lambda_hx_h}\Biggr)\Biggl(\prod_{\substack{h=0\\h\not =k}}^{D} \frac{\lambda_h}{(1+\lambda_hx_h)^2}\Biggr)\sum_{n=1}^\infty n^{D+1-j}(n+1)^{j-1}\prod_{h=0}^{D}\Biggl(\frac{\lambda_hx_h}{1+\lambda_hx_h} \Biggr)^{n-1}. \label{leggeCicloCiclicoMinimaleGeometrico}
\end{align}
Similarly, formula (\ref{formulaCiclicoMinimaleGeneraleCaso0}) reads
\begin{align}
P\bigg\{&X(t)\in \dif x,\,\bigcup_{n=1}^\infty N(t) =n(D+1)-1,\,V(t) = e_k\bigg\} /\dif x\nonumber\\
& = \frac{p_{k+1}}{1+\lambda_k x_k}\Biggl(\prod_{\substack{h=0\\h\not =k}}^{D} \frac{\lambda_h}{(1+\lambda_hx_h)^2}\Biggr)\sum_{n=0}^\infty (n+1)^{D}\prod_{h=0}^{D}\Biggl(\frac{\lambda_hx_h}{1+\lambda_hx_h} \Biggr)^{n}.\label{leggeCicloCiclicoMinimaleGeometricoCaso0}
\end{align}
Finally, for $x\longrightarrow \bar{x}\in \partial \text{Supp}\bigl(X(t)\bigr)$ similar considerations to those in point $(a)$ apply.

We point out that from the above formulas it is easy to obtain several results appeared in previous papers such as \cite{DcIM2023, L2006, LLO2006, IV2023, O2002}. 
%Here the authors usually consider a slightly different probability, by conditioning on the initial velocity ($V(0) = v_0$). This distribution follows from (\ref{formulaCiclicoMinimaleGenerale}) (and then from (\ref{leggeCicloCiclicoMinimale}) or (\ref{leggeCicloCiclicoMinimaleGeometrico})) by considering only the term for $k = j-1$ (indeed, the probability with the event $\{V(t)=v_{j-1}\}$ joint to $N(t) = n(D+1)+j-1$, is equivalent to evaluate the probability joint to $V(0) = v_0$) and from (\ref{formulaCiclicoMinimaleGeneraleCaso0}) (and then from (\ref{leggeCicloCiclicoMinimaleCaso0}) and (\ref{leggeCicloCiclicoMinimaleGeometricoCaso0})) by considering the term with $k=D$ (indeed, if $V(t) = v_D$ and $N(t) = n(D+1)-1$, then $V(0) = v_0$).} 
For instance, by considering $\lambda_h = \lambda>0\ \forall\ h$ and $k = j-1$, (\ref{leggeCicloCiclicoMinimale}) coincides with the distribution in Section 4.4 of \cite{L2006}; with $D=1$, from formulas (\ref{leggeCicloCiclicoMinimaleGeometrico}) and (\ref{leggeCicloCiclicoMinimaleGeometricoCaso0}) it is straightforward to derive the elegant distributions in Theorem 1 of \cite{DcIM2023} (consider $k=j-1=0$ in  (\ref{leggeCicloCiclicoMinimaleGeometrico}) and $k = D=1$ in (\ref{leggeCicloCiclicoMinimaleGeometricoCaso0})).

For further details about the cyclic motions we refer to \cite{L2006} and \cite{LLO2006}. 
\hfill$\diamond$
\end{example}

\begin{example}[Complete motions]\label{esempioLeggeMotoCanonicoCompleto}
Let $X$ be a $D$-dimensional complete canonical (minimal) random motion with $P\{V(0) = e_h\}=P\{V(t+\dif t) = e_h\,|\, V(t)=e_j,\, N(t, t+\dif t] = 1\} = p_{h}> 0$ for each $j,h = 0,\dots, D,$ and governed by a homogeneous Poisson process with rate $\lambda>0$. Now, with $t\ge0$, in light of Remark \ref{remarkLeggeMotoCanonico}, by setting $x_0 = t-\sum_{j=0}^D x_j$ and using Theorem \ref{teoremaLeggeIntersezioniMinimale} we readily arrive at, for $x \in \overset{\circ}{\text{Supp}}\bigl(X(t)\bigr) $, integers $n_0,\dots,n_D\ge 1$ and $k=0,\dots,D$,
\begin{align}
& P\bigg\{X(t)\in \dif x,\,\bigcap_{h=0}^D\{N_h(t) = n_h\} ,\, V(t) = e_k\bigg\}/\dif x\nonumber\\
& =\Biggl( \prod_{\substack{h=0\\h\not=k}}^D \frac{\lambda^{n_h}\,x_h^{n_h-1}\,e^{-\lambda x_h}}{(n_h-1)!}\Biggr)\, \frac{e^{-\lambda x_k}(\lambda x_k)^{n_k-1}}{(n_k-1)!} \, \binom{n_0+\dots+n_D-1}{n_0,\dots,n_{k-1},n_{k}-1,n_{k+1},\dots, n_D} \prod_{h=0}^D p_h^{n_h}\nonumber\\
& =\frac{e^{-\lambda t}}{\lambda}\,\biggl(\,\sum_{h=0}^D n_h -1\biggr)! \,n_k\prod_{h=0}^D \frac{(\lambda p_h)^{n_h}\, x_h^{n_h-1}}{(n_h-1)!\,n_h!}.\label{leggeIntersezioniMinimaleCompleto}
\end{align}
Then it is straightforward to see that,
\begin{equation}
P\bigg\{X(t)\in \dif x,\,\bigcap_{h=0}^D\{N_h(t) = n_h\}\bigg\} /\dif x= \frac{e^{-\lambda t}}{\lambda}\,\biggl(\,\sum_{h=0}^D n_h \biggr)! \,\prod_{h=0}^D \frac{(\lambda p_h)^{n_h}\, x_h^{n_h-1}}{(n_h-1)!\,n_h!}.\label{leggeIntersezioniMinimaleCompleto2}
\end{equation}
Finally, 
\begin{align}
P\{X(t)\in \dif x\}/\dif x &= \frac{e^{-\lambda t}}{\lambda} \sum_{n_0,\dots,n_D \ge1}\, \biggl(\,\sum_{h=0}^D n_h \biggr)! \,\prod_{h=0}^D \frac{(\lambda p_h)^{n_h}\, x_h^{n_h-1}}{(n_h-1)!\,n_h!} \label{leggeMotoMinimaleCompleto}\\
& = \frac{e^{-\lambda t}}{\lambda}\sum_{m_0,\dots,m_D \ge0}  \,\biggl(\,\sum_{h=0}^D m_h + D+1 \biggr)! \,\prod_{h=0}^D \frac{(\lambda p_h)^{m_h+1}\, x_h^{m_h}}{m_h!\,(m_h+1)!}\nonumber \\
& = \frac{e^{-\lambda t}}{\lambda} \prod_{h=0}^D \sqrt{\lambda p_h} \,\sum_{m_0,\dots,m_D \ge0} \,\int_0^\infty e^{-w}   w^{D+1}\,\prod_{h=0}^D \frac{(\lambda p_h)^{m_h+\frac{1}{2}}\, (x_h w)^{m_h}}{m_h!\,(m_h+1)!}   \dif w\nonumber\\
 & =  \frac{e^{-\lambda t}}{\lambda} \prod_{h=0}^D \sqrt{\frac{\lambda p_h}{x_h}} \int_0^\infty  e^{-w}   w^{\frac{D+1}{2}} \prod_{h=0}^D I_1\Bigl(2\sqrt{w\lambda p_h x_h}\Bigr)\dif w,\label{leggeMotoMinimaleCompletoIntegrale}
\end{align}
with $I_1(z) = \sum_{n=0}^\infty \Bigl(\frac{z}{2}\Bigr)^{2n+1} \frac{1}{n!\,(n+1)!}$ is the modified Bessel function of order $1$, for $ z\in \mathbb{C}$. Note that, if $x\longrightarrow \bar{x}\in \partial \text{Supp}\bigl(X(t)\bigr)$, then $\exists \ l=0,\dots, D$ such that $x_l\longrightarrow0$. For instance, by assuming that there is just one $l$ satisfying the given condition, formula (\ref{leggeMotoMinimaleCompletoIntegrale}) turns into
$$ P\{X(t)\in \dif x\}/\dif x  \longrightarrow p_l e^{-\lambda t}\prod_{\substack{h=0\\h\not=l}}^D \sqrt{\frac{\lambda p_h}{x_h}} \int_0^\infty  e^{-w}   w^{\frac{D}{2}+1} \prod_{\substack{h=0\\h\not=l}}^D I_1\Bigl(2\sqrt{w\lambda p_h x_h}\Bigr)\dif w.$$
Similarly to the cyclic case (see point $(a)$, limit behavior of (\ref{formulaCiclicoMinimaleGeneraleCaso0})), probability (\ref{leggeMotoMinimaleCompletoIntegrale}) never converges to $0$ because we are including the event where each velocity is chosen once. This can be easily observed from formula (\ref{leggeIntersezioniMinimaleCompleto}) by putting $n_l= 1$.

It is interesting to observe that,
\begin{align}
\int_{\text{Supp}\bigl(X(t)\bigr)} P\{X(t)\in \dif x\}& = \frac{e^{-\lambda t}}{\lambda} \sum_{n_0,\dots,n_D \ge1}\, \biggl(\,\sum_{h=0}^D n_h \biggr)! \,\prod_{h=0}^D \frac{(\lambda p_h)^{n_h}}{(n_h-1)!\,n_h!}\nonumber \\
&\ \ \ \times  \int_0^{t} x_1^{n_1-1} \dif x_1\int_0^{t-x_1} x_2^{n_2-1}\dif x_2\dots \int_{0}^{t-x_1-\dots-x_{D-2}} x_{D-1}^{n_{D-1}-1}\dif x_{D-1} \nonumber\\
& \ \ \ \times\int_0^{t-x_1-\dots-x_{D-1}} x_D^{n_D-1} \biggl(t-\sum_{j=1}^D x_j\biggr)^{n_0-1} \dif x_D\nonumber\\
& = \frac{e^{-\lambda t}}{\lambda t} \sum_{n_0,\dots,n_D \ge1}\,\biggl(\,\sum_{h=0}^D n_h \biggr) \prod_{h=0}^D \frac{(\lambda t p_h)^{n_h}}{n_h!}\nonumber\\
& = e^{-\lambda t} \sum_{h=0}^D p_h \sum_{n_0,\dots,n_D \ge1}\,  \frac{(\lambda t p_h )^{n_h-1}}{(n_h-1)!} \prod_{\substack{j=0\\j\not=h}}^D \frac{(\lambda t p_j )^{n_j}}{n_j!}\nonumber\\
&= e^{-\lambda t} \sum_{h=0}^D p_h e^{\lambda t p_h} \prod_{\substack{j=0\\j\not=h}}^D \bigl(e^{\lambda t p_j}-1\bigr)\nonumber\\
& = 1-P\bigg\{\bigcup_{h=0}^D \{N_h(t) =0\}\bigg\}. \label{integraleLeggeMotoMinimaleCanonicoCompleto}
\end{align}
For the details about the last equality see Appendix \ref{masseProabilitaMotoCanonicoCompleto}. If $p_0 = \dots = p_D = 1/(D+1)$, probability (\ref{integraleLeggeMotoMinimaleCanonicoCompleto}) reduces to $e^{\frac{-\lambda t D}{D+1}} \bigl(e^{\frac{-\lambda t }{D+1}}-1\bigr)^D$.

To obtain the distribution of the position of an arbitrary $D$-dimensional complete minimal random motion governed by a homogeneous Poisson process, we easily use the above probabilities and relationship (\ref{X'funzioneXcanonico}).
\hfill$\diamond$
\end{example}

\subsubsection{Distribution on the boundary of the support}\label{sottosezioneSingolaritaMotiMinimali}
Let $X$ be a minimal random motion with velocities $v_0,\dots,v_D$. Theorem \ref{teoremaLeggeIntersezioniMinimale} describes the joint probability in the inner part of the support of the position $X(t)$, i.e. Conv$(v_0t,\dots, v_Dt), t\ge0$. Now we deal with the distribution over the boundary of $\text{Supp}\bigl(X(t)\bigr)$, which can be partitioned in $\sum_{H=0}^{D-1} \binom{D+1}{H+1}$ components, corresponding to those in (\ref{supportoMotoMinimale}) with $H<D$.
\\

Fix $H \in\{ 0,\dots,D-1\}$ and let $I_H = \{i_0,\dots,i_H\}\in \mathcal{C}_{H+1}^{\{0,\dots,D\}}$ be a combination of $H+1$ indexes in $\{0,\dots, D\}$. At time $t\ge0$, the motion $X$ lies on the set $\overset{\circ}{\text{Conv}}(v_{i_0}t, \dots, v_{i_H}t)$ if and only if it moves with all and only the velocities $v_{i_0},\dots, v_{i_H}$ in the time interval $[0,t]$. Hence, if $X(t)\in \overset{\circ}{\text{Conv}}(v_{i_0}t, \dots, v_{i_H}t)\ a.s.$ we can write the following, for $k=0,\dots,H$,
\begin{equation}\label{XbordoFunzioneT}
X(t) = \sum_{h=0}^H v_{i_h}T_{(i_h)}(t) = v_{i_k}t +\sum_{\substack{h=0\\h\not=k}}^H (v_{i_h}-v_{i_k}) T_{(i_h)}(t) = g_k(T_{(i_k^-)}^H(t)) ,
\end{equation}
where $ T_{(\cdot)}^{H}(t) = \bigl(T_{(i_0)}(t),\dots, T_{(i_H)}(t)\bigr)$ and $ T_{(i_k^-)}^{H}(t) = \Bigl(T_{(i_h)}(t)\Bigr)_{\substack{h=0, \dots, H\\h\not=k}}$.  The function $g_k:[0,+\infty)^H\longrightarrow \mathbb{R}^D$ in (\ref{XbordoFunzioneT}) is an affine relationship. 

By keeping in mind that $v_0,\dots, v_D$ are affinely independent, then $\text{dim}\Bigl(\text{Conv}(v_{i_0}t, \dots, v_{i_H}t)\Bigr) = H$ and, from Lemma \ref{lemmaProiezione}, there exists an orthogonal projection onto a $H$-dimensional space, $p_H:\mathbb{R}^D\longrightarrow\mathbb{R}^H$, such that $v_{i_0}^H = p_H(v_{i_0}),\dots,v_{i_H}^H = p_H(v_{i_H})$ are affinely independent and that we can characterize the vector $X(t)$, when lying on the set $\overset{\circ}{\text{Conv}}(v_{i_0}t, \dots, v_{i_H}t)\ a.s.$, through its projection $X^H(t) = p_H\bigl(X(t)\bigr)$. Hence, we just need to study the projected motion
%rank$(v_{i_0} \ \cdots\ v_{i_H})\ge H$, thus $v_{i_0}, \dots, v_{i_H}$ are $H$-affinely independent. This means that there exists an orthogonal projection onto a $H$-dimensional space, $p_H:\mathbb{R}^D\longrightarrow\mathbb{R}^H$, such that $v_{i_0}^H = p_H(v_{i_0}),\dots,v_{i_H}^H = p_H(v_{i_H})$ that are affinely independent. With this at hand, we can characterize the vector $X(t)$, when lying on the set $\overset{\circ}{\text{Conv}}(v_{i_0}t, \dots, v_{i_H}t)\ a.s.$ through its projection $X^H(t) = p_H\bigl(X(t)\bigr)$ such that
\begin{equation}\label{XHfunzioneTH}
X^H(t) = v_{i_k}^H t+\sum_{\substack{h=0\\h\not=k}}^H (v_{i_h}^H-v_{i_k}^H) T_{(i_h)}(t) = g_k^H\Bigl(T_{(i_k^-)}^H(t)\Bigr), \ \  \ t\ge0,\, k=0,\dots, H.
\end{equation}
It is straightforward to see that the vector $X^{H^-}(t)$ containing the components of $X(t)$ that are not included in $X^H(t)$, is such that, for $x\in\overset{\circ}{\text{Conv}}(v_{i_0}t, \dots, v_{i_H}t)$ with $x^H = p_H(x)\in \mathbb{R}^H$ and $x^{H^-}\in \mathbb{R}^{D-H}$ denoting the other entries of $x$,
\begin{equation}
P\bigg\{X^{H^-}(t)\in \dif y\,\Big|\, X^H(t) = x^H,\, \bigcap_{i\in\{0,\dots, D\}\setminus I_H} \{N_i(t) = 0\}\bigg\} = \delta \bigl(y-x^{H^-}\bigr)\dif y,
\end{equation}
with $y\in  \mathbb{R}^{D-H}$ and $\delta$ the Dirac delta function centered in $0$.

Now, the function $g_k^H:\mathbb{R}^H\longrightarrow \mathbb{R}^H$ in (\ref{XHfunzioneTH}) is a bijection and we can write, $\forall \ k$,
\begin{equation}\label{inversaXHfunzioneTH}
T^H_{(i_k^-)}(t) = \bigl(g_k^H\bigr)^{-1}\Bigl( X^H(t)\Bigr) = \Biggl(\bigl(g_k^H\bigr)_h^{-1}\Bigl( X^H(t)\Bigr)\Biggr)_{\substack{h=0,\dots, H\\h\not=k}} = \Biggl[v_{i}^H - v_{i_k}^H\Biggr]^{-1}_{\substack{i\in I_H\\i\not=i_k}}\,\Bigl(X^H(t)-v_{i_k}^Ht\Bigr).
\end{equation}
Note that formula (\ref{inversaXHfunzioneTH}) coincides with (\ref{inversatXfunzioneT}) if $H = D$.

\begin{theorem}\label{teoremaLeggeSingolaritaIntersezioniMinimale}
Let $X$ be a minimal finite-velocity random motion in $\mathbb{R}^D$ satisfying (H1)-(H2). Let $H = 0,\dots, D-1$ and $I_H = \{i_0,\dots,i_H\}\in \mathcal{C}_{H+1}^{\{0,\dots,D\}}$. Then the orthogonal projection $p_H:\mathbb{R}^D\longrightarrow\mathbb{R}^H$ defined in Lemma \ref{lemmaProiezione} (there $p_R$) exists and $v_{i_0}^H = p_H(v_{i_0}),\dots,v_{i_H}^H = p_H(v_{i_H})$ are affinely independent. Furthermore, for $t\ge0,\,x\in \overset{\circ}{\text{Conv}}(v_{i_0}t, \dots, v_{i_H}t),\ n_{i_0},\dots,n_{i_H}\in \mathbb{N}$ and $k=0,\dots, H$,
\begin{align}
& P\bigg\{X(t)\in \dif x,\,\bigcap_{h=0}^H\{N_{i_h}(t) = n_{i_h}\},\,\bigcap_{i\in I_{H^-}}\{N_{i}(t) = 0\} ,\, V(t) = v_{i_k}\bigg\}/ \dif x\label{leggeSingolaritaIntersezioniMinimale}\\
& = \prod_{\substack{h=0\\h\not=k}}^H f_{\sum_{j=1}^{n_{i_h}} W_j^{(i_h)}}\biggl(\bigl(g_{k}^H\bigr)_h^{-1}\bigl(x^H\bigr)\biggr) \,\Bigg|\biggl[v_{i}^H - v_{i_k}^H\biggr]_{\substack{i\in I_H\\i\not=i_k}}\Bigg|^{-1} \, P\Bigg\{N_{(i_k)}\biggl(t-\sum_{\substack{h=0\\h\not = k}}^H \bigl(g_{k}^H\bigr)_h^{-1}\bigl(x^H\bigr)\biggr) = n_{i_k}-1\Bigg\}\nonumber\\
&\ \ \ \times P\big\{C_{n_{i_0}+\dots+n_{i_H}} = (n_0,\dots,n_D), V(t) = v_{i_k}\big\},\nonumber
\end{align}
where $x^H = p_H(x)$, $\bigl(g_k^H\bigr)^{-1}$ is given in (\ref{inversaXHfunzioneTH}), $I_{H^- } = \{0,\dots, D\}\setminus I_H$ and suitable $n_0,\dots,n_D$.
\end{theorem}

Note that the projection defined in Lemma \ref{lemmaProiezione} is usually not the only suitable one.
\begin{proof}
In light of the considerations above, the proof follows equivalently to the proof of Theorem \ref{teoremaLeggeIntersezioniMinimale}
\end{proof}

\begin{remark}[Canonical motion]\label{remarkTeoriaMotoCanonicoSingolarita}
Let $X$ be a canonical (minimal) random motion, governed by a point process $N$, and $I_H = \{i_0,\dots,i_H\}\in \mathcal{C}_{H+1}^{\{0,\dots,D\}},\ H=0,\dots, D-1$. We build the projection $p_H$ such that it selects the first $H$ linearly independent rows of $(e_{i_0} \ \cdots\ e_{i_H})$, if $i_0 = 0$, and the last ones if $i_0\not=0$. Then, $(e_{i_0}^H \ \cdots\ e_{i_H}^H) = (0\ I_H)$ and, by proceeding as shown in Remark \ref{remarkDefMotoCanonico}, we obtain $T_{(\cdot)}^H(t) = \bigl(t-\sum_{h=1}^H X_{i_h}^H(t), X^H(t)\bigr)$; note that in this case the indexes of the velocities ($i_1,\dots,i_H$) coincide with the indexes of the selected coordinates of the motion

Now, if $Y$ is a minimal random motion with velocities $v_0,\dots, v_D$ and governed by $N_Y \stackrel{d}{=}N$, for each $I_H = \{i_0,\dots,i_H\}\in \mathcal{C}_{H+1}^{\{0,\dots,D\}},\ H=0,\dots, D-1$, by using the arguments leading to	(\ref{X'funzioneXcanonico}), we can write
\begin{equation}\label{X'funzioneXsingolarita}
Y^H(t) \stackrel{d}{=} v_{i_0}^H t +\biggl[v_{i}^H - v_{i_k}^H\biggr]_{\substack{i\in I_H\\i\not=i_k}} X^H(t).
\end{equation}
We point out that the motions are related through the times of the displacements with each velocity and not directly through their coordinates. This means that $X^H$ and $Y^H$ are not necessarily obtained through the same projection, but they are respectively related to the processes $T_{(\cdot)}^H$ and $T_{(\cdot)}^{Y,H}$ that have the same finite dimensional distributions since $N_Y \stackrel{d}{=}N$ (see proof of Theorem \ref{teoremaX'funzioneX}).
\hfill$\diamond$
\end{remark}

Note that Remark \ref{remarkTeoriaMotoCanonicoSingolarita} holds even though hypothesis (H1)-(H2) are not assumed.
\\

By comparing Theorem \ref{teoremaLeggeIntersezioniMinimale} with Theorem \ref{teoremaLeggeSingolaritaIntersezioniMinimale}, we note that there is a strong similarity between the distribution of a $D$-dimensional minimal motion over its singularity of dimension $H$ (in fact, dim$\Bigl(\text{Conv}(v_{i_0}t, \dots, v_{i_H}t)\Bigr)=H,\ t>0$) and the distribution of an $H$-dimensional minimal motion moving with velocities $v_{i_0}^H=p_H(v_{i_0}), \dots, v_{i_H}^H = p_H(v_{i_H})$. These kind of relationships are further investigated in the next sections (see also the next example); in particular, Theorem \ref{teoremaSingolaritaMotiPoisson} states a result concerning a wide class of random motions.

\begin{example}[Complete motions: distribution over the singular components]\label{esempioSingolaritaMotoMinimaleCompleto}
Let us consider the complete canonical random motion $X$ studied in Example \ref{esempioLeggeMotoCanonicoCompleto}. Let $I_H = \{i_0,\dots,i_H\}\in \mathcal{C}_{H+1}^{\{0,\dots,D\}}$ and $I_{H^-} = \{0,\dots, D\}\setminus I_H$, with $H=0,\dots, D-1$, we now compute the probability density of being in $x\in\overset{\circ}{\text{Conv}}(e_{i_0}t, \dots, e_{i_H}t)$ at time $t\ge0$. By keeping in mind Theorem \ref{teoremaLeggeSingolaritaIntersezioniMinimale} and Remark \ref{remarkTeoriaMotoCanonicoSingolarita} (and by proceeding as shown for probability (\ref{leggeIntersezioniMinimaleCompleto})), for integers $n_{i_0},\dots ,n_{i_H}\ge1$ and $k = 0,\dots,H$, we have that
\begin{align*}%\label{probabilitaSingolaritaCanonico}
 P\bigg\{&X(t)\in \dif x,\,\bigcap_{h=0}^H\{N_{i_h}(t) = n_{i_h}\},\,\bigcap_{i\in I_{H^-}}\{N_{i}(t) = 0\} ,\, V(t) = e_{i_k}\bigg\}/ \dif x \nonumber\\
&= \frac{e^{-\lambda t}}{\lambda}\,\biggl(\,\sum_{h=0}^H n_{i_h} -1\biggr)! \,n_{i_k}\prod_{h=0}^H \frac{(\lambda p_{i_h})^{n_{i_h}}\, x_{i_h}^{n_{i_h}-1}}{(n_{i_h}-1)!\,n_{i_h}!}
\end{align*}
where $x_{i_0} = t-\sum_{j=1}^H x_{i_j}$. Clearly, by working as shown in Example \ref{esempioLeggeMotoCanonicoCompleto}, we obtain
\begin{align}
P\bigg\{X(t)\in \dif x,\,\bigcap_{i\in I_{H^-}}\{N_{i}(t) = 0\}\bigg\}&/\dif x = \frac{e^{-\lambda t}}{\lambda} \sum_{n_{i_0},\dots,n_{i_H} \ge1}\, \biggl(\,\sum_{h=0}^H n_{i_h} \biggr)! \,\prod_{h=0}^H \frac{(\lambda p_{i_h})^{n_{i_h}}\, x_{i_h}^{n_{i_h}-1}}{(n_{i_h}-1)!\,n_{i_h}!} \nonumber\\%\label{leggeSingolaritaMotoMinimaleCompleto}\\
 & =  \frac{e^{-\lambda t}}{\lambda} \prod_{h=0}^H \sqrt{\frac{\lambda p_{i_h}}{x_{i_h}}} \int_0^\infty  e^{-w}   w^{\frac{H+1}{2}} \prod_{h=0}^H I_1\Bigl(2\sqrt{w\lambda p_{i_h} x_{i_h}}\Bigr)\dif w\nonumber%\label{leggeSingolaritaMotoMinimaleCompletoIntegrale}
\end{align}
and
\begin{align}
\int_{\text{Conv}(e_{i_0}t,\dots, e_{i_H}t)}&P\bigg\{X(t)\in \dif x,\,\bigcap_{i\in I_{H^-}}\{N_{i}(t) = 0\}\bigg\} =   e^{-\lambda t} \sum_{h=0}^H p_{i_h} e^{\lambda t p_{i_h}} \prod_{\substack{j=0\\j\not=h}}^H \bigl(e^{\lambda t p_{i_j}}-1\bigr)\nonumber\\
& = P\bigg\{\bigcap_{i\in I_{H^-}}\{N_{i}(t) = 0\}\bigg\}-P\bigg\{\bigcup_{i\in I_H}\{N_{i}(t)=0\},\,\bigcap_{i\in I_{H^-}}\{N_{i}(t) = 0\}\bigg\},\label{massaSingolaritaMotoCanonicoCompleto}
\end{align}
where further details about the last equality are in Appendix \ref{masseProabilitaMotoCanonicoCompleto}.
\\

Let $Y$ be a complete minimal motion governed by a counting process $N_Y \stackrel{d}{=}N$ and moving with velocities $v_0,\dots, v_D$. By suitably applying relationship (\ref{X'funzioneXsingolarita}) and the above probabilities it is easy to obtain the distribution of the position $Y(t)$ over its singular components.
\hfill$\diamond$
\end{example}

\section{Random motions with a finite number of velocities}

\begin{proposition}\label{proposizioneMotiAVelocitaFinita}
Let $X$ be a random motion governed by a point process $N$ and moving with velocities $v_0,\dots, v_M\in \mathbb{R}^D, \ M\in\mathbb{N}$, such that $\text{dim}\Bigl(\text{Conv}(v_{0},\dots,v_{M})\Bigr) = R\le D$. Then, the orthogonal projection $p_R:\mathbb{R}^D\longrightarrow\mathbb{R}^R$ defined in Lemma \ref{lemmaProiezione} exists and, for $t\ge0$, we can characterize $X(t)$ through its projection $X^R(t) = p_R\bigl(X(t)\bigr)$, representing the position of a $R$-dimensional motion moving with velocities $p_R(v_0), \dots, p_R(v_M)$ and governed by $N$.
\end{proposition}

\begin{proof}
The projection $p_R$ exists since the hypothesis of Lemma \ref{lemmaProiezione} are satisfied. By keeping in mind the characteristics of $p_R$ (see Lemma \ref{lemmaProiezione}), we immediately obtain that $\forall\ A\subset \text{Conv}(v_{0},\dots,v_{M})$ and its projection through $p_R$, $A^R \subset \text{Conv}\bigl(p_R(v_{0}),\dots,p_R(v_{M})\bigr)$, $\{\omega\in \Omega\,:\,X(\omega,t)\in A\} = \{\omega\in \Omega\,:\,X^R(\omega,t)\in A^R\}$.
\end{proof}

Proposition \ref{proposizioneMotiAVelocitaFinita} states that if $\text{dim}\Bigl(\text{Conv}(v_{0},\dots,v_{M})\Bigr) = R\le D$, then we can equivalently study the process $X$, a random motion with $M+1$ velocities in $\mathbb{R}^D$, or its projection $X^R$, a random motion of $M+1$ velocities in $\mathbb{R}^R$. This means that we can limit ourselves to the study of random motions where the dimension of the space coincides with the dimension of the state space. Clearly, for $R=D$ Proposition \ref{proposizioneMotiAVelocitaFinita} is not of interest since $p_R$ is the identity function.

\begin{remark}[Motions with affinely independent velocities]\label{remarkMotiVelocitaAffini}
Let $X$ be a random motion moving with affinely independent velocities $v_0,\dots, v_H\in\mathbb{R}, \ H\le D$. In light of Proposition \ref{proposizioneMotiAVelocitaFinita}, $\exists$ the orthogonal projection $p_H$ given in Lemma \ref{lemmaProiezione} such that to study $X^H = \big\{p_H\bigl(X(t)\bigr)\big\}_{t\ge0}$ is equivalent to study $X$. The process $X^H$ is a minimal random motion moving with velocities $p_H(v_0), \dots, p_H(v_H)$ and, if it satisfies (H1)-(H2), theorems \ref{teoremaLeggeIntersezioniMinimale} and \ref{teoremaLeggeSingolaritaIntersezioniMinimale} provide its probability law.
\hfill$\diamond$
\end{remark}

% NOTA: questo esempio potrebbe anche essere tolto
\begin{example}[Motion with canonical velocities]
Let $X$ a $D$-dimensional motion moving with the first $H$ canonical velocities $e_0,\dots,e_H$ and satisfying (H1)-(H2). For $t\ge0$, Supp$\bigl(X(t)\bigr) = \{x\in \mathbb{R}^D\,:\, x\ge0,\ x_{H+1},\dots,x_D = 0,\ \sum_{i=1}^H x_i = t\}$ and, by following the arguments of Section \ref{sottosezioneSingolaritaMotiMinimali} we can derive the probability distribution of $X(t)$ in the inner part of its support by using formula (\ref{leggeSingolaritaIntersezioniMinimale}), that uses the connection to the projected position $p_H\bigl(X(t)\bigr)$. In this case, the last probability of (\ref{leggeSingolaritaIntersezioniMinimale}) becomes $P\big\{C_{n_0+\dots+n_H} = (n_0,\dots,n_H), V(t) = v_{i_k}\big\}$ with $n_0,\dots, n_H\not=0,$ and therefore it coincides with the probability of the $H$-dimensional canonical motion.
\hfill$\diamond$
\end{example}

\subsection{Motions in $\mathbb{R}^D$ with $D$-dimensional state space}

Thanks to Proposition \ref{proposizioneMotiAVelocitaFinita} and Remark \ref{remarkMotiVelocitaAffini}, in order to cover the analysis of all the possible motions (under the given assumptions), we need to deal with random motions in $\mathbb{R}^D$ moving with $M+1$ velocities, $M>D$, and with state space of dimension $D$.

%Let $X$ be a random motion in $\mathbb{R}^D$ with $D$-dimensional state space, in details, $X$ moves with velocities $v_0,\dots, v_M\in \mathbb{R}^D$, with $M>D$, such that $\text{dim}\Bigl(\text{Conv}(v_{0},\dots,v_{M})\Bigr) = D$. We study $X$ by observing that $\exists$ a minimal  random motion $\tilde{X}$ in $\mathbb{R}^M$ such that $X$ is the marginal vector process of $\tilde{X}$ represented by its first $D$ components, i.e. $X(t) = (I_D \ 0)\tilde{X}(t),\ t\ge0$.
%\\Let $v_{}$ if $x\in \text{Conv}(v_{0},\dots,v_{M})$}

\begin{proposition}\label{propMotiNonMinimali}
Let $X$ be a random motion governed by a point process $N$ and moving with velocities $v_0,\dots, v_M\in \mathbb{R}^D, \ D<M\in\mathbb{N}$, such that $\text{dim}\Bigl(\text{Conv}(v_{0},\dots,v_{M})\Bigr) = D$. Then, there exists a minimal  random motion $\tilde{X}$ in $\mathbb{R}^M$ such that $X$ is the marginal vector process of $\tilde{X}$ represented by its first $D$ components.
\end{proposition}

\begin{proof}
Let $V,N$ be the processes respectively governing the velocity and the displacements of $X$. Let $\pi_D:\mathbb{R}^M\longrightarrow \mathbb{R}^D, \ \pi_D(\tilde{x}) = (I_D\ 0)\tilde{x},\, \tilde{x}\in\mathbb{R}^M$. Then, $\exists\ \tilde{v}_0,\dots, \tilde{v}_M\in \mathbb{R}^M$ affinely independent such that $\pi_D(\tilde{v}_h) = v_h\ \forall \ h$. The random motion $\tilde{X}$ with displacements governed by $N$ and velocity process $\tilde{V}$, with state space $\{\tilde{v_0},\dots,\tilde{v_M}\}$ and such that $\pi_D\bigl(\tilde{V}(t)\bigr) = V(t)$ (i.e. $\{\tilde{V}(t) = \tilde{v}_h\} \iff \{V(t) = v_h\}\ \forall\ h,t$), is a minimal random motion in $\mathbb{R}^M$ and $\pi_D\bigl(\tilde{X}(t)\bigr) = X(t)\ \forall\ t$.
\end{proof}

From the proof of Proposition \ref{propMotiNonMinimali} it is obvious that $\exists$ infinite $M$-dimensional stochastic motions $\tilde{X}$ of the required form.

\begin{remark}\label{distribuzioneMotoGenerale}[Distribution of the position of the motion]
Let $X$ be a random motion with velocities $v_0,\dots, v_M\in \mathbb{R}^D, \ M\in\mathbb{N}$, such that $\text{dim}\Bigl(\text{Conv}(v_{0},\dots,v_{M})\Bigr) = D$. In light of Proposition \ref{propMotiNonMinimali}, we provide the distribution of $X(t),\, t\ge0,$ in terms of the probabilities of the position of minimal random motions. 

Let $\tilde{X}$ be a minimal random motion such as in Proposition \ref{propMotiNonMinimali} and $\pi_D$ the orthogonal projection in the proof above. Now, for $t\ge0,\,x\in \overset{\circ}{\text{Conv}}(v_{0}t, \dots, v_{M}t)$, natural $n_0,\dots,n_M\ge1$ and $k=0,\dots, M$, we can write
\begin{align}
 P&\bigg\{X(t)\in \dif x,\,\bigcap_{h=0}^M\{N_h(t) = n_h\} ,\, V(t) = v_k\bigg\}\label{leggeIntersezioniNonMinimale}\\
& = \int_{A_x}P\bigg\{\tilde{X}(t)\in \dif (x,y),\,\bigcap_{h=0}^M\{N_h(t) = n_h\} ,\, \tilde{V}(t) = \tilde{v}_k\bigg\} \nonumber
\end{align}
where $A_x = \big\{y\in \mathbb{R}^{M-D}\,:\, (x,y)\in \text{Conv}(\tilde{v}_{0}t,\dots,\tilde{v}_{M}t)\big\}$; clearly, $\pi_D(x,y) =(I_D \ 0)(x,y)= x$. Under assumptions (H1)-(H2), probability (\ref{leggeIntersezioniNonMinimale}) can be written explicitly by means of Theorem \ref{teoremaLeggeIntersezioniMinimale}.

Remember that, differently from the minimal motion case, the support of $X(t)$ is not partitioned by the elements appearing in (\ref{supportoMotoMinimale}) (since they are not disjoint). Thus, for fix $t\ge0$ and $x\in \text{Conv}(v_{0}t, \dots, v_{M}t)$ there may exist several combinations of velocities (and their corresponding times) such that the motion is in position $x$ at time $t$. With $H=1,\dots, M$, let $ I_{x,t,H}^{(1)},\dots, I_{x,t,H}^{(L_H)}\in \mathcal{C}^{\{0,\dots, M+1\}}_{H+1}$ be the $L_H\le\binom{M+1}{H+1}$ possible combinations of $H+1$ velocities such that the motion can lie in $x$ at time $t$, i.e. $x\in \overset{\circ}{\text{Conv}}(v_{i_0}t, \dots, v_{i_H}t)$ with $i_0\dots,i_H\in  I_{x,t,H}^{(l)}, \ \forall \ l,H$ (clearly, for some $H$ it can happen that there are no suitable combinations in $\mathcal{C}^{\{0,\dots, M+1\}}_{H+1}$, so $L_H = 0$). In general, we can write (omitting the indexes $x,t$ of $I_{x,t,H}^{(l)}$),
\begin{align}
&P\{X(t)\in \dif x\}/\dif x \nonumber \\
&=\sum_{k=0}^M P\bigg\{X(t)\in \dif x,\,\bigcup_{H=1}^M\bigcup_{l=1}^{L_H} \Big\{ \bigcap_{i\in I_{H}^{(l)}}\{N_i(t)\ge1\},\, \bigcap_{i\in I_{H^-}^{(l)}}\{N_i(t) = 0\}\Big\},\, V(t) = v_k\bigg\} /\dif x\nonumber\\
& = \sum_{k=0}^M\sum_{H=1}^M\sum_{l=1}^{L_H} \sum_{\substack{n_h=1\\h\in I_{H}^{(l)}}}^\infty  P\bigg\{X(t)\in \dif x,\, \bigcap_{i\in I_{H}^{(l)}}\{N_i(t)=n_i\},\, \bigcap_{i\in I_{H^-}^{(l)}}\{N_i(t) = 0\},\, V(t) = v_k\bigg\}  /\dif x\label{formulaGeneraleMotiNonMinimali}
\end{align}
where $I_{H^-}^{(l)} = \{0,\dots, M\}\setminus I_H^{(l)},\ \forall\ l,H$.
\\Now, under the hypothesis (H1)-(H2), the probabilities appearing in (\ref{formulaGeneraleMotiNonMinimali}) can be obtained by using previous results. Consider the combination of velocities $I_{H}^{(l)} =\{i_0,\dots, i_H\}$:
\begin{itemize}
\item[($a$)] if $\text{dim}\Bigl(\text{Conv}(v_{i_0},\dots,v_{i_H})\Bigr) = H (\le D)$ then we can compute the corresponding probability in (\ref{formulaGeneraleMotiNonMinimali}) by suitably using Theorem \ref{teoremaLeggeSingolaritaIntersezioniMinimale} (if $H=D$, then the projection described in Theorem \ref{teoremaLeggeSingolaritaIntersezioniMinimale} turns into the identity function).

\item[($b$)] if $\text{dim}\Bigl(\text{Conv}(v_{i_0},\dots,v_{i_H})\Bigr) = R < H$, then we use the following argument. In light of Proposition \ref{proposizioneMotiAVelocitaFinita}, we can consider the orthogonal projection $p_R$ defined in Lemma \ref{lemmaProiezione} and we study the process $X^R$ with velocities $v_{i_0}^R = p_R(v_{i_0}),\dots, v_{i_H}^R = p_R(v_{i_H})$. Then, $X^R$ is a $R$-dimensional motion with $H+1$ velocities and we can proceed as shown for probability (\ref{leggeIntersezioniNonMinimale}). Let us denote with $\tilde{X}^R$ the minimal motion such that $\pi_R\bigl(\tilde{X}^R(t)\bigr) =(I_R\ 0)\tilde{X}(t) = X^R(t), \ t\ge0$, and with $\tilde{V}^R$ the corresponding velocity process, with state space $\{\tilde{v}^R_{i_0},\dots, \tilde{v}_{i_H}^R\},$ where $\pi(\tilde{v}^R_{i_h}) = v_{i_h}^R\ \forall\ h$. Now, for $n_{i_0}, \dots, n_{i_H}\in \mathbb{N}$ and $k=0,\dots, H$,
\begin{align}
&P\bigg\{X(t)\in \dif x,\, \bigcap_{i\in I_{H}^{(l)}}\{N_i(t)=n_i\},\, \bigcap_{i\in I_{H^-}^{(l)}}\{N_i(t) = 0\},\, V(t) = v_k\bigg\}/ \dif x\label{formulaGeneraleMotiNonMinimali2}\\
& = \int_{A_{x}}P\bigg\{\tilde{X}^R(t)\in \dif (x^R,y),\, \bigcap_{i\in I_{H}^{(l)}}\{N_i(t)=n_i\},\, \bigcap_{i\in I_{H^-}^{(l)}}\{N_i(t) = 0\},\, \tilde{V}^R(t) = \tilde{v}^R_{i_k}\bigg\}  /\dif x^R \nonumber
\end{align}
where $A_x = \{y\in \mathbb{R}^{H-R}\,:\, (x^R,y) \in \text{Conv}(\tilde{v}_{i_0}t,\dots,\tilde{v}_{i_H}t)\big\} $ and clearly $\pi_R(x^R,y) = x^R$.\hfill$\diamond$
\end{itemize}
\end{remark}

\begin{example}
Let $X$ be a one-dimensional cyclic motion moving with velocities $v_0 = 0, v_1 = 1, v_2 = -1$ and $p_h = P\{V(0) = v_h\}>0\ \forall\ h$. Let $N$ be its governing Poisson-type process such that $W^{(h)}_j\sim Exp(\lambda_h),\ h=0,1,2, \ j\in\mathbb{N}$. We now consider the two-dimensional minimal random motion $(X,Y)$ moving with velocities $\tilde{v}_0 = (0,1), \tilde{v}_1 =(1, 0),  \tilde{v}_2 = (-1,0)$ governed by $N$. Let $t\ge0$ and $x\in (0,t)$. In order to reach $x$ the motion must perform at least one displacements with $v_1$. Thus, by keeping in mind the cyclic routine for the velocities ($\dots \rightarrow v_0\rightarrow v_1\rightarrow v_2\rightarrow\dots$), the probability reads
\begin{align}
&P\{X(t)\in \dif x\}= \nonumber\\
&P \{X(t)\in \dif x,\, N_0(t) = 1,\, N_1(t) = 1,\, N_2(t) = 0\} + P \{X(t)\in \dif x,\, N_0(t) = 0,\, N_1(t) = 1,\, N_2(t) = 1\} \nonumber\\
&\ \ \  + \sum_{j=0}^{2} P \bigg\{X(t)\in \dif x,\, \bigcup_{n=1}^\infty N(t) = 3n+j-1\bigg\}\nonumber\\
& = P\{W_1^{(0)}\in \dif(t- x),\,V(0) = v_0\} + P\{W_1^{(1)}\in \dif \frac{(t+x)}{2},V(0) = v_1\}\nonumber\\
&\ \ \ +\sum_{j=0}^{2}\,\int_0^{t-x} P\bigg\{X(t)\in \dif x,\, Y(t)\in \dif y,\,\bigcup_{n=1}^\infty N(t) = 3n+j-1\bigg\}.\label{integraleMotoNonMinimale}
\end{align}
The first two terms are respectively given by $p_0 \lambda_0e^{-\lambda_0 (t-x)}\dif x$ and $p_1\lambda_1 e^{-\lambda_1 \frac{t+x}{2}}\dif x$. By suitably applying Theorem \ref{teoremaLeggeIntersezioniMinimale} or Example \ref{remarkLeggeMotoCiclico}, the interested reader can explicitly compute (\ref{integraleMotoNonMinimale}). Note that the integral in (\ref{integraleMotoNonMinimale}) is of the form $\int_0^{t-x} y^{n_0} \Bigl(\frac{t+x-y}{2}\Bigr)^{n_1}\Bigl(\frac{t-x-y}{2}\Bigr)^{n_2} \dif y$ with suitable natural $n_0,n_1,n_2$.
\hfill$\diamond$
\end{example}

\section{Random motions governed by non-homogeneous Poisson process}\label{sezioneMotiGovernatiPoisson}

Here we consider a random motion $X$ moving with a natural number of finite velocities $v_0,\dots, v_M\in\mathbb{R}^D, M\in \mathbb{N}$, whose movements are governed by a non-homogeneous Poisson process $N$ with rate function $\lambda:[0,\infty)\longrightarrow[0,\infty)$. In this case $N$ cannot explode in a bounded time interval if and only if $\Lambda(t) = \int_0^t \lambda(s)\dif s < \infty, t\ge0$. We note that the process $X$ satisfies (H2) if and only if $\lambda(t) = \lambda>0\ \forall\ t$.

Let us assume that, $\forall\ t,\, p_i =P\{V(0) = v_i\}$ and $p_{i,j}=P\{V(t+\dif t) = v_j\,|\,  V(t)=v_i,\,N(t, t+\dif t] = 1\} \ge 0$ for each $i,j = 0,\dots, M$. Let us also consider the notation, with $t\ge0, x\in \text{Supp}\bigl(X(t)\bigr)$,
$$p(x,t)\dif x = P\{X(t)\in \dif x\} = \sum_{i=0}^M P\{X(t)\in \dif x,\, V(t) = v_i\} = \sum_{i=0}^M f_i(x,t)\dif x.$$
It can be proved that the functions $f_i$ satisfy the differential problem (with $<\cdot,\cdot>$ denoting the dot product in $\mathbb{R}^D$)
\begin{equation}\label{sistemafi}
\begin{cases}
\frac{\partial f_i}{\partial t} = -<\nabla_x f_i, v_i > -\lambda(t)f_i +\lambda(t)\sum_{j=0}^M p_{j,i}f_j,\ \ \ i=0,\dots, M,\\
f_i(x,t)\ge0,\ \ \  \forall\ i,x,t,\\
\int_{\text{Conv}(v_0t,\dots,v_Mt)} \sum_{i=0}^M f_i(x,t)\dif x = 1 - P\Big\{ \bigcup_{h=0}^M \{N_h(t) = 0\}\Big\},
\end{cases}
\end{equation}
where $\nabla_x f$ represents the $x$-gradient vector of $f$ and $P\Big\{ \bigcup_{h=0}^M \{N_h(t) = 0\}\Big\}>0 \iff \Lambda(t)<\infty\ \forall\ t$. We refer to \cite{CO2023, CO2022, KT1998, O1990} for proofs similar to the one leading to (\ref{sistemafi}).

\begin{remark}[Complete minimal motions]
Let $X$ be a complete canonical (minimal) random motion (see Example \ref{esempioLeggeMotoCanonicoCompleto}) such that $\forall\ i,j,\ p_{i,j} = p_j$. The differential problem governing the probability law of $X$ satisfies
\begin{equation}\label{sistemaMotoCompletoMinimale}
\begin{cases}
\displaystyle\frac{\partial f_0}{\partial t} =\lambda(t)p_0 \sum_{j=1}^D f_j + \lambda(t)(p_0-1)f_0,\\
\displaystyle\frac{\partial f_i}{\partial t} =  -\frac{\partial f_i}{\partial x_i} + \lambda(t)p_i\sum_{\substack{j=0\\j\not=i}}^D f_j + \lambda(t)(p_i-1)f_i,\ \  i=1,\dots,D,\\
f_i(x,t)\ge0, \ \ \forall \ i,x,t, \ \ \ \int_{\text{Supp}\bigl(X(t)\bigr)} \sum_{i=0}^D f_i(x,t)\dif x = 1-P\bigg\{\bigcup_{h=0}^D \{N_h(t) =0\}\bigg\}.
\end{cases}
\end{equation}
Through a direct calculation, it is easy to show that the probabilities obtained by suitably adapting distributions (\ref{leggeIntersezioniMinimaleCompleto}), i.e. by summing with respect to $n_0,\dots,n_D\ge1$, satisfy the partial differential equations in (\ref{sistemaMotoCompletoMinimale}) with $\lambda(t) = \lambda>0\ \forall\ t$. Furthermore, as shown in Example \ref{esempioLeggeMotoCanonicoCompleto}, the sum of these probabilities, i.e. (\ref{leggeMotoMinimaleCompleto}), satisfies the condition in system (\ref{sistemaMotoCompletoMinimale}) (see (\ref{integraleLeggeMotoMinimaleCanonicoCompleto})).

It is also possible to show that, if $\lambda(t) = \lambda>0, \ \forall\ t$, probability (\ref{leggeMotoMinimaleCompleto}) (that is $p=\sum_i f_i$) is solution to the following $D$-th order partial differential equation
\begin{equation}\label{relazioneDifferenzialeMotoCompletoMinimale}
\sum_{k=0}^D \sum_{i\in \mathcal{C}_k^{\{1,\dots,D\}}} \sum_{h=0}^{D+1-k} \lambda^{D+1-(h+k)} \Biggl[ \binom{D+1-k}{h} - \Bigl(p_0+\sum_{j\not \in i} p_j\Bigr)\binom{D-k}{h}\Biggr] \frac{\partial^{h+k} p}{\partial t^{h}\partial x_{i_1}\cdots \partial x_{i_k}} = 0
\end{equation}
The proof of this result is shown in Appendix \ref{appendicePDEMotoMinimaleCompleto}.
\hfill$\diamond$
\end{remark}

The next statement concerns the distribution over the singular components when $N$ cannot explode in finite time intervals.

\begin{theorem}\label{teoremaSingolaritaMotiPoisson}
Let $X$ be a finite-velocity random motion moving with velocities $v_0,\dots, v_M\in\mathbb{R}^D,\ M\in \mathbb{N}$, governed by a non-homogeneous Poisson process $N$ with rate function $\lambda\in C^{M}\bigl([0,\infty),[0,\infty)\bigr)$ such that $\Lambda(t) = \int_0^t \lambda(s)\dif s < \infty, t\ge0$. Let $p_i = P\{V(0) = v_i\}>0$ and $p_{i,j} = P\big\{V(t+\dif t) = v_j\,|\,V(t)=v_i,\, N(t, t+\dif t] = 1\big\} \ge 0$ for each $i,j = 0,\dots, M,\,\forall\ t$. 
\\Set $H = 0,\dots, M-1$, $I_H = \{i_0,\dots,i_H\}\in \mathcal{C}_{H+1}^{\{0,\dots,M\}}$ and $I_{H^- } = \{0,\dots, M\}\setminus I_H$. If
\begin{equation}\label{ipotesiTransizioniPoisson}
\sum_{j\in I_H} p_{i_k,j} = P\big\{V(t+\dif t)\in \{v_{i_0}, \dots,v_{i_H}\}\,|\, V(t) = v_{i_k}, N(t,t+\dif t] = 1\big\} = \alpha_{I_H}>0
\end{equation}
for $k =0,\dots, H$ and $t\ge0$. Then, with $\text{dim}\Bigl(\text{Conv}(v_{i_0},\dots,v_{i_H})\Bigr) = R\le D$, $\exists$ an orthogonal projection $p_R:\mathbb{R}^D\longrightarrow\mathbb{R}^R$ such that, for $t\ge0,\, x\in \overset{\circ}{\text{Conv}}(v_{i_0}t, \dots, v_{i_H}t) $, with $x^R = p_R(x)$,
\begin{align}
P\bigg\{X(t)\in \dif x\,\Big|\,\bigcap_{j\in I_{H^-}}\{N_{j}(t) = 0\}\bigg\}/\dif x = P\Big\{Y^R(t)\in \dif x^R\Big\}/\dif x^R,\label{proiezioneSingolaritaNonOmogeneo}
\end{align}
where $Y^R$ is a $R$-dimensional finite-velocity random process governed by a non-homogeneous Poisson process with rate function $\lambda\alpha_{I_H}$, moving with velocities $v_{i_0}^R =p_R(v_{i_0}),\dots, v_{i_H}^R =p_R(v_{i_H})$ and such that $ p^Y_i = p_i/\sum_{j\in I_H} p_j$ and $p^Y_{i,j} = p_{i,j}/\alpha_{I_H}\ \forall\ i,j\in I_H$.
\end{theorem}

Theorem \ref{teoremaSingolaritaMotiPoisson} states that if the probability of keeping a velocity with index in $I_H$ is constant ($\alpha_{I_H}$), then, with respect to the conditional measure $P\Big\{\,\cdot\,|\,\bigcap_{j\in I_{H^-}}\{N_j(t) = 0\}\Big\}$, $X$ is equal in distribution (in terms of finite dimensional distributions) to a $R$-dimensional motion governed by a non-homogeneous Poisson process with rate function $\lambda\alpha_{I_H}$ and suitably scaled transition probabilities, where $R = \text{dim}\Bigl(\text{Conv}(v_{i_0},\dots,v_{i_H})\Bigr)$ (if $R = D$, the identity function fits $p_R$).

\begin{proof}
First, we note that, in light of (\ref{ipotesiTransizioniPoisson}), for $t\ge0$, $P\big\{V(t+\dif t)\in \{v_{i_0},\dots, v_{i_M}\}\,|\,\, V(t)\in \{v_{i_0},\dots, v_{i_M}\},\,N(t, t+\dif t]=1\big\} = \alpha_{I_H}$, and thus
\begin{align}
P\bigg\{\bigcap_{j\in I_{H^-}}\{N_{j}(t) = 0\}\bigg\}& = P\big\{V(0)\in \{v_{i_0},\dots, v_{i_M}\}\big\}\sum_{n=0}^\infty P\{ N(t) = n\} \, \alpha_{I_H}^n\nonumber\\
& = e^{-\Lambda(t)(1-\alpha_{I_H})}\sum_{i\in I_H} p_i. \label{formulaProbabilitaSingolaritaNonOmogeneo}
\end{align}
Now, from Proposition \ref{proposizioneMotiAVelocitaFinita}, Lemma \ref{lemmaProiezione} and the same argument used in point ($b$) of Remark \ref{distribuzioneMotoGenerale}, there exists a projection $p_R:\mathbb{R}^D\longrightarrow\mathbb{R}^R$ such that $X^R(t) = p_R\bigl(X(t)\bigr)$ and $P\Big\{X(t)\in \dif x\,\Big|\,\bigcap_{j\in I_{H^-}}\{N_{j}(t) = 0\}\Big\}/\dif x = P\Big\{X^R(t)\in \dif x^R\,\Big|\,\bigcap_{j\in I_{H^-}}\{N_{j}(t) = 0\}\Big\}/\dif x^R $, with $x\in \overset{\circ}{\text{Conv}}(v_{i_0}t, \dots, v_{i_H}t)$. The $R$-dimensional motion $X^R$ moves with velocities $v_{0}^R =p_R(v_0),\dots, v_M^R = p_R(v_M)$ and its probability functions
$$f_i(y, t)\dif y = P\Big\{X^R(t)\in \dif y,\, \bigcap_{j\in I_{H^-}}\{N_{j}(t) = 0\},\,V_{X^R}(t) = v_{i}^R\Big\},\ \  i\in I_H,$$
with $t\ge0, \,y \in \overset{\circ}{\text{Conv}}\bigr(v_{i_0}^Rt,\dots,v_{i_H}^Rt\bigr)$, satisfy the differential system
\begin{equation}\label{sistemafiSingolarita}
\begin{cases}
\displaystyle\frac{\partial f_i}{\partial t} = -<\nabla_y f_i, v_i^R > -\lambda(t)f_i +\lambda(t)\sum_{j\in I_H} p_{j,i}f_j,\ \ \ i\in I_H,\\
f_i(y,t)\ge0,\ \ \   i\in I_H,\, \forall\ y,t,\\
\displaystyle\int_{\text{Conv}(v_{i_0}^Rt,\dots,v_{i_H}^Rt)} \sum_{i\in I_H} f_i(y,t)\dif y= \int_{\text{Conv}(v_{i_0}^Rt,\dots,v_{i_H}^Rt)}P\bigg\{X^R(t)\in \dif y,\, \bigcap_{j\in I_{H^-}}\{N_{j}(t) = 0\}\bigg\}\\
\displaystyle\hspace{2.8cm}= P\bigg\{\bigcap_{j\in I_{H^-}}\{N_{j}(t) = 0\}\bigg\} - P\bigg\{\bigcup_{i\in I_H} \{N_i(t) = 0\},\, \bigcap_{j\in I_{H^-}} \{N_j(t) = 0\}\bigg\},
\end{cases}
\end{equation}
In light of (\ref{formulaProbabilitaSingolaritaNonOmogeneo}) we consider $f_i(y,t) = g_i(y,t) e^{-\Lambda(t)(1-\alpha_{I_H})}\sum_{h\in I_H} p_h,\ \forall\ i$, i.e. $g_i(y,t)\dif y = P\Big\{X^R(t)\in \dif y,\,V_{X^R}(t) = v_{i}^R\,\Big|\, \bigcap_{j\in I_{H^-}}\{N_{j}(t) = 0\}\Big\}$. System (\ref{sistemafiSingolarita}) becomes 
\begin{equation}\label{sistemafiSingolaritaTrasformato}
\begin{cases}
\displaystyle\frac{\partial g_i}{\partial t} = -<\nabla_y\, g_i, v_i^R > -\lambda(t)\alpha_{I_H}g_i +\lambda(t)\alpha_{I_H}\sum_{j\in I_H} \frac{p_{j,i}}{\alpha_{I_H}}g_j,\ \ \ i\in I_H,\\
g_i(y,t)\ge0,\ \ \  i\in I_H,\,\forall\ y,t,\\
\begin{split}
\displaystyle\int_{\text{Conv}(v_{i_0}^Rt,\dots,v_{i_H}^Rt)} \sum_{i\in I_H} g_i(y,t)\dif y&= \int_{\text{Conv}(v_{i_0}^Rt,\dots,v_{i_H}^Rt)}P\bigg\{X^R(t)\in \dif y\,\Big|\, \bigcap_{j\in I_{H^-}}\{N_{j}(t) = 0\}\bigg\}\\
\displaystyle&= 1 - P\bigg\{\bigcup_{i\in I_H} \{N_i(t) = 0\}\,\Big|\, \bigcap_{j\in I_{H^-}} \{N_j(t) = 0\}\bigg\},
\end{split}
\end{cases}
\end{equation}
which coincides with the system satisfied by the distribution of the position of the stochastic motion $Y^R$ in the statement.
\end{proof}

Theorems 3.1 and 3.2 of Cinque and Orsingher \cite{CO2022} are particular cases of Theorem \ref{teoremaSingolaritaMotiPoisson}.

\appendix

\section{Appendix. Proof of Lemma \ref{lemmaProiezione}}\label{dimLemmaProiezione}

If $\,\text{dim}\Bigl(\text{Conv}(v_0,\dots,v_M)\Bigr) = R$, the matrix $\mathrm{V}_{(k)} = \Bigl[v_h-v_k\Bigr]_{\substack{h=0,\dots,M\\h\not=k}}$ has $R$ linearly independent rows $\forall\ k$. Now, the matrix $\mathrm{V}_{(k)}^R = \Bigl[v_h^R-v_k^R\Bigr]_{\substack{h=0,\dots,M\\h\not=k}}$, obtained by keeping the first $R$ linearly independent rows of $\mathrm{V}_{(k)}$, has rank $R$ and therefore $\text{dim}\Bigl(\text{Conv}\bigl(v_0^R,\dots,v_M^R\bigr)\Bigr) = R$. Thus, $\forall\ l$, also $\mathrm{V}_{(l)}^R = \Bigl[v_h^R-v_l^R\Bigr]_{\substack{h=0,\dots,M\\h\not=l}}$ has rank $R$ and these must be the first $R$ linearly independent rows of $\mathrm{V}_{(l)}$ (if not, by proceeding as above for $k$, we would obtain that the $R$ selected rows were not the first linearly independent ones for $V_{(k)}$ which is a contradiction). 
\\Finally, the second part of the lemma follows from the equivalence of the linear systems
\begin{equation}\label{sistemiPerLemma}
\ \text{ and }\ 
\end{equation} 
where $a=(a_0,\dots,a_M)\in\mathbb{R}^{M+1}$, such that $a_i\in[0,1]\ \forall\ i$ and $\sum_{i=0}^M a_i = 1$, is the unknown variable.
Indeed, for $k=0,\dots,M$, thanks to the constraints on $a$, the systems in (\ref{sistemiPerLemma}) can be written as
$$ x - v_k = \Bigl[v_h-v_k\Bigr]_{h\not=k}a_{(k)}\ \text{ and }\ x^{R} - v_k^{R} = \Bigl[v_h^{R}-v_k^{R}\Bigr]_{h\not=k} a_{(k)},$$
with $a_{(k)}=(a_0,\dots,a_{k-1},a_{k+1},\dots, a_M)$.

\section{Appendix. Complete canonical random motion}

Let $X$ be a complete canonical random motion as in Example \ref{esempioLeggeMotoCanonicoCompleto}.

\subsection{Probability mass of the singularity}\label{masseProabilitaMotoCanonicoCompleto}

Before computing the probability mass of the singularities of the complete uniform random motion, we need to show some useful relationships.

Let $c_1, \dots c_H\in \mathbb{R}, \ H\in\mathbb{N}$ and $\mathcal{C}_h^{\{1, \dots, H\}}$ the combinations of $h$ elements among $\{1, \dots, H\}$, $h=1,\dots,H$. We have that
\begin{equation}\label{primaRelazioneUtile}
\sum_{h=1}^H (-1)^{H-h} \sum_{i\in \mathcal{C}_h^{\{1, \dots, H\}}} (c_{i_1}+\dots c_{i_h})^m =
\begin{cases}
\begin{array}{l l}
 0, &\ m<H,\\
\displaystyle \sum_{\substack{n_1,\dots,n_H\ge1\\n_1+\dots+n_H=m}}  c_1^{n_1}\cdots c_H^{n_H}\binom{m}{n_1,\dots,n_H}, &\ m \ge H.
\end{array}
\end{cases}
\end{equation}
%and, for $m\ge H$,
%\begin{equation}\label{secondaRelazioneUtile}
%\sum_{j=1}^H (-1)^{H-j}\sum_{i\in \mathcal{C}_j^{\{1, \dots, H\}}} (c_{i_1}+\dots c_{i_j})^m  = \sum_{\substack{n_1,\dots,n_H\ge1\\n_1+\dots+n_H=m}}  c_1^{n_1}\cdots c_H^{n_H}\binom{m}{n_1,\dots,n_H}.
%\end{equation}
and, with $\beta \in \mathbb{R}$,
\begin{equation}\label{terzaRelazioneUtile}
\sum_{h=1}^H c_h e^{\beta c_h}\prod_{\substack{j = 1\\j\not=h}}^{H} \bigl(e^{\beta c_j}-1\bigr) = \sum_{h=1}^{H} (-1)^{H-h} \sum_{i\in \mathcal{C}_h^{\{1,\dots, H\}}} (c_{i_1}+\dots +c_{i_h})\,e^{\beta(c_{i_1}+\dots+c_{i_h})}.
\end{equation}

To prove (\ref{primaRelazioneUtile}), we denote with $\mathcal{C}_{h,\{i_1,\dots,i_j\}}^{\{1,\dots,H\}} $ the combinations of $h$ elements in $\{1,\dots,H\}$ containing $i_1,\dots,i_j$, with $1\le j\le h\le H$ and suitable $i_1,\dots,i_j$. Then
\begin{align}
\sum_{h=1}^H& (-1)^{H-h} \sum_{i\in \mathcal{C}_h^{\{1, \dots, H\}}} (c_{i_1}+\dots c_{i_h})^m \nonumber\\
&= \sum_{h=1}^H (-1)^{H-h} \sum_{i\in \mathcal{C}_h^{\{1, \dots,H\}}} \sum_{\substack{n_1,\dots,n_h\ge 0\\n_1+\dots+n_h = m} } c_{i_1}^{n_1}\cdots c_{i_h}^{n_h}\binom{m}{n_1,\dots,n_h}\label{passaggioBase}\\
& = \sum_{j=1}^m \sum_{k\in \mathcal{C}_j^{\{1,\dots,H\}}}  \sum_{\substack{m_1,\dots,m_j\ge 1\\m_1+\dots+m_j= m}} c_{k_1}^{m_1}\cdots c_{k_j}^{m_j}\binom{m}{m_1,\dots,m_j} \sum_{h=j}^H (-1)^{H-h}\, \Big| \mathcal{C}_{h,\{k_1,\dots,k_j\}}^{\{1,\dots,H\}}\Big| \label{passaggioParticolare}\\
& = \sum_{j=1}^m \sum_{k\in \mathcal{C}_j^{\{1,\dots,H\}}}  \sum_{\substack{m_1,\dots,m_j\ge 1\\m_1+\dots+m_j= m}} c_{k_1}^{m_1}\cdots c_{k_j}^{m_j}\binom{m}{m_1,\dots,m_j} \,(-1)^{H+j}\sum_{l=0}^{H-j} (-1)^l \binom{H-j}{l} \label{relazioneUtileUltimoPassaggio}\\
&=\begin{cases}
\begin{array}{l l}
 0, &\ m<H,\\
\displaystyle \sum_{\substack{n_1,\dots,n_H\ge1\\n_1+\dots+n_H=m}}  c_1^{n_1}\cdots c_H^{n_H}\binom{m}{n_1,\dots,n_H}, &\ m \ge H.
\end{array}
\end{cases}\nonumber
\end{align}
In fact, in (\ref{relazioneUtileUltimoPassaggio}), the last sum (with index $l$) is equal to $0$ for $j\not= H$ and $1$ for $j = H$.
In step (\ref{passaggioParticolare}) we express (\ref{passaggioBase}) by summing every possible combination of indexes ($k_1,\dots,k_j$) and every possible allocation of exponents ($m_1,\dots,m_j\ge1, m_1+\dots+m_j=m$). Each of these elements, $c_{k_1}^{m_1}\cdots c_{k_j}^{m_j}$, appears one time in the expansion of $(c_{i_1}+\dots+c_{i_h})^m$ for each $i\in \mathcal{C}_{h,\{k_1,\dots,k_j\}}^{\{1,\dots,H\}}$, with $1\le j \le h\le H$, i.e. $\Big| \mathcal{C}_{h,\{k_1,\dots,k_j\}}^{\{1,\dots,H\}}\Big| =\binom{H-j}{h-j}$ times.

To prove (\ref{terzaRelazioneUtile}) we proceed as follows, by denoting with $\mathcal{C}^{\{1,\dots,H\}}_{k,(h)}$ the combinations of $k$ elements not containing $h$,
\begin{align}
\sum_{h=1}^H c_h e^{\beta c_h}\prod_{\substack{j = 1\\j\not=h}}^{H} \bigl(e^{\beta c_j}-1\bigr)& =\sum_{h=1}^H c_h e^{\beta c_h} \sum_{k=0}^{H-1} (-1)^{H-1-k} \sum_{i\in \mathcal{C}^{\{1,\dots,H\}}_{k,(h)}} e^{\beta(c_{i_1}+\dots+c_{i_k})}\nonumber\\
&= \sum_{k=0}^{H-1} (-1)^{H-1-k} \sum_{h=1}^H c_h \sum_{i\in \mathcal{C}^{\{1,\dots,H\}}_{k,(h)}} e^{\beta(c_h+ c_{i_1}+\dots+c_{i_k})}\label{passaggioTerzaRelazioneUtile}\\
& =  \sum_{k=0}^{H-1} (-1)^{H-1-k} \sum_{i\in \mathcal{C}^{\{1,\dots,H\}}_{k+1}} (c_{i_1}+\dots+c_{i_{k+1}})e^{\beta(c_{i_1}+\dots+c_{i_{k+1}})}\nonumber
\end{align}
which coincides with (\ref{terzaRelazioneUtile}). The last step follows by observing that for each combination $i\in \mathcal{C}^{\{1,\dots,H\}}_{k+1}$ the corresponding exponential term appears once for each $h\in i=(i_1,\dots,i_{k+1})$, with $h$ being the index of the second sum of (\ref{passaggioTerzaRelazioneUtile}).
\\

We now compute the probability mass that the motion moves with all and only $H+1$ precise velocities, for $H = 0,\dots,D-1$. Let $I_H = \{i_0,\dots, i_H\} \in \mathcal{C}_{H+1}^{\{0,\dots,D\}}$, then
\begin{align}
P\Big\{&X(t)\in \overset{\circ}{\text{Conv}}(v_{i_0}t, \dots, v_{i_H}t)\Big\} = P\bigg\{ \bigcap_{i\in I_H} \{N_i(t) \ge1\},\,\bigcap_{i\not\in I_H} \{N_i(t) =0\} \bigg\}\label{probabilitaVelocitaEsatte}\\
& = \sum_{n=H}^{\infty} P\{N(t) = n \}  \sum_{\substack{n_0,\dots, n_H\ge1\\n_0+\dots+n_H=n+1}} p_{i_0}^{n_0}\cdots p_{i_H}^{n_H} \binom{n+1}{n_0,\dots, n_H}\nonumber\\
& = \sum_{h=1}^{H+1} (-1)^{H+1-h} \sum_{i\in \mathcal{C}_h^{\{0, \dots, H\}}} \,\sum_{n=H}^{\infty} P\{N(t) = n\}(p_{i_0}+\dots +p_{i_h})^{n+1} \label{usoPrimaRelazioneUtileProbabilita}\\
& = \sum_{h=1}^{H+1} (-1)^{H+1-h} \sum_{i\in \mathcal{C}_h^{\{0, \dots, H\}}} (p_{i_0}+\dots +p_{i_h}) e^{-\lambda t(1-p_{i_0}-\dots- p_{i_h})} \label{connessioneFormulaEsponenziali}  \\
& \ \ \ - e^{-\lambda t} \sum_{n=0}^{H-1} \frac{(\lambda t)^n}{n!} \, \sum_{h=1}^{H+1} (-1)^{H+1-h} \sum_{i\in \mathcal{C}_h^{\{0, \dots, H\}}} (p_{i_0}+\dots +p_{i_h})^{n+1} \label{termineNulloMassaSingolarita} \\
& = (p_{0}+\dots +p_{H})\, e^{-\lambda t(1-p_{0}-\dots- p_H)} - \sum_{h=1}^{H} (-1)^{H-h} \sum_{i\in \mathcal{C}_h^{\{0, \dots, H\}}} (p_{i_0}+\dots +p_{i_h})\, e^{-\lambda t(1-p_{i_0}-\dots- p_{i_h})}  \nonumber\\
& = P\bigg\{ \bigcap_{i\not\in I_H} \{N_i(t) =0\} \bigg\} - P\bigg\{ \bigcup_{i\in I_H} \{N_i(t) =0\},\,\bigcap_{i\not\in I_H} \{N_i(t) =0\} \bigg\} \nonumber
\end{align}
where we used the second equality of (\ref{primaRelazioneUtile}) to derive (\ref{usoPrimaRelazioneUtileProbabilita}). Thanks to the first case of (\ref{primaRelazioneUtile}), it is easy to see that the term (\ref{termineNulloMassaSingolarita}) is $0$ and thus, by means of (\ref{terzaRelazioneUtile}), we also obtain the equivalence between (\ref{connessioneFormulaEsponenziali}) and the probability mass (\ref{massaSingolaritaMotoCanonicoCompleto}).
\\Note that if the motion is \textit{uniform}, i.e. $p_0 = \dots = p_D = 1/(D+1)$, probability (\ref{probabilitaVelocitaEsatte}) reduces to $ \frac{H+1}{D+1} e^{\frac{-\lambda t D}{D+1}} \bigl(e^{\frac{-\lambda t }{D+1}}-1\bigr)^H$ (see also (\ref{massaSingolaritaMotoCanonicoCompleto})).

In light of (\ref{connessioneFormulaEsponenziali}), the probability that the motion moves with exactly $H+1$ velocities in the time interval $[0,t]$ is 
\begin{align}
P&\Bigg\{\bigcup_{I \in \mathcal{C}_{H+1}^{\{0,\dots, D\}}} \bigg \{\bigcap_{i\in I}\{N_i(t)\ge1\},\, \bigcap_{i\not\in I}\{ N_i(t) = 0\} \bigg\} \Bigg\} =\sum_{I \in \mathcal{C}_{H+1}^{\{0,\dots, D\}}} P\bigg\{ \bigcap_{i\in I} \{N_i(t) \ge1\},\,\bigcap_{i\not\in I} \{N_i(t) =0\} \bigg\}\nonumber\\
&= \sum_{h=1}^{H+1} (-1)^{H+1-h} \sum_{I \in \mathcal{C}_{H+1}^{\{0,\dots, D\}}}\sum_{i\in \mathcal{C}_h^{I}} (p_{i_0}+\dots +p_{i_h}) \,e^{-\lambda t(1-p_{i_0}-\dots- p_{i_h})}\nonumber\\
& = \sum_{h=1}^{H+1} (-1)^{H+1-h} \binom{D+1-h}{H+1-h}\sum_{i \in \mathcal{C}_{h}^{\{0,\dots, D\}}} (p_{i_0}+\dots +p_{i_h}) \,e^{-\lambda t(1-p_{i_0}-\dots- p_{i_h})}\label{massaH+1Velocita}
\end{align}
where in the last step we observe that each combination $i \in \mathcal{C}_{h}^{\{0,\dots, D\}}$ appears in $ \binom{D+1-h}{H+1-h}$ combinations in $\mathcal{C}_{H+1}^{\{0,\dots, D\}}$ (i.e. all those which contain the $h$ elements in $i$).

Finally, by using expression (\ref{massaH+1Velocita}),
\begin{align}
&P\big\{X(t)\in \partial \text{Supp}\bigl(X(t)\bigr)\big\} = P\bigg \{\bigcup_{h=0}^D\{ N_h(t) = 0\} \bigg\} \nonumber\\
& =  P\Bigg\{\,\bigcup_{H=0}^{D-1}\,\bigcup_{I \in \mathcal{C}_{H+1}^{\{0,\dots, D\}}} \bigg \{\bigcap_{i\in I}\{N_i(t)\ge1\},\, \bigcap_{i\not\in I}\{ N_i(t) = 0\} \bigg\} \Bigg\}\nonumber \\
&=\sum_{H=0}^{D-1} \sum_{h=1}^{H+1} (-1)^{H+1-h} \binom{D+1-h}{H+1-h}\sum_{i \in \mathcal{C}_{h}^{\{0,\dots, D\}}} (p_{i_0}+\dots +p_{i_h}) \,e^{-\lambda t(1-p_{i_0}-\dots- p_{i_h})}\nonumber\\
& = \sum_{h=1}^{D}  \sum_{i \in \mathcal{C}_{h}^{\{0,\dots, D\}}} (p_{i_0}+\dots +p_{i_h}) \,e^{-\lambda t(1-p_{i_0}-\dots- p_{i_h})} \sum_{H=h-1}^{D-1}  (-1)^{H+1-h} \binom{D+1-h}{H+1-h}\nonumber\\
& = \sum_{h=1}^{D}  \sum_{i \in \mathcal{C}_{h}^{\{0,\dots, D\}}} (p_{i_0}+\dots +p_{i_h}) \,e^{-\lambda t(1-p_{i_0}-\dots- p_{i_h})} \Bigl(0-(-1)^{D+1-h}\Bigr)\nonumber\\
& = \sum_{h=1}^{D} (-1)^{D-h} \sum_{i \in \mathcal{C}_{h}^{\{0,\dots, D\}}} (p_{i_0}+\dots +p_{i_h}) \,e^{-\lambda t(1-p_{i_0}-\dots- p_{i_h})}.\label{connessioneFormulaEsponenzialiBordo}
\end{align}
Note that, with (\ref{connessioneFormulaEsponenzialiBordo}) at hand, by also keeping in mind (\ref{terzaRelazioneUtile}) and that $p_0+\dots+p_D = 1$, we obtain the last step in probability (\ref{integraleLeggeMotoMinimaleCanonicoCompleto})
\\

It is interesting to observe that if the point process governing $X$ is a non-homogeneous Poisson process with rate function $\lambda:[0,\infty)\longrightarrow [0,\infty)$ such that $\Lambda(t) = \int_0^t \lambda(s)\dif s <\infty\ \forall \ t$, then, the above probability masses hold with $\Lambda(t)$ replacing $\lambda t$.

\subsection{PDE governing the absolutely continuous component}\label{appendicePDEMotoMinimaleCompleto}

From the differential system (\ref{sistemaMotoCompletoMinimale}) we obtain (\ref{relazioneDifferenzialeMotoCompletoMinimale}) through the following iterative argument. 
\\First, we consider $w_1 = f_0 + f_1$ and we easily obtain
\begin{align}\label{derivataPrimatw1}
\frac{\partial w_1}{\partial t} &= \lambda(p_0+p_1-1)w_1- \frac{\partial f_1}{\partial x_1} + \lambda(p_0+p_1)\sum_{j=2}^D f_j\nonumber\\
& = A w_1 + B f_1 + C \sum_{j=2}^D f_j,
\end{align}
with $A,B,C$ suitable operators. Now, we rewrite the equations of (\ref{sistemaMotoCompletoMinimale}) by means of the operators $E_i = \Bigl(\frac{\partial }{\partial x_i} + \lambda\Bigr)$ and $G_i = \lambda p_i$,
\begin{align}\label{rappresentazioneSistemaConOperatori}
\frac{\partial f_i}{\partial t} = -E_i f_i + G_i \sum_{j=0}^i f_j + G_i \sum_{j=i+1}^D f_j, \ \ i=1,\dots D.
\end{align}

By keeping in mind (\ref{derivataPrimatw1}), (\ref{rappresentazioneSistemaConOperatori}) (for $i=1$) and the exchangeability of the differential operators, we can express the second-order time derivative of $w_1$ in terms of $w_1$ and $\sum_{j=2}^D f_j$, 
\begin{align}
 \frac{\partial^2 w_1}{\partial t^2} & =  A \frac{\partial w_1}{\partial t} + B \frac{\partial f_1}{\partial t} + C \frac{\partial}{\partial t}  \sum_{j=2}^D f_j \nonumber\\
& = A \frac{\partial w_1}{\partial t}   + B \Biggl(-E_1 f_1 + G_1 w_1 + G_1  \sum_{j=2}^D f_j \Biggr) + C \frac{\partial}{\partial t} \sum_{j=2}^D f_j\nonumber \\
& = \Biggl( A \frac{\partial }{\partial t} + B G_1  \Biggr) w_1 -E_1 \Biggl(\frac{\partial w_1}{\partial t} - A w_1 -C \sum_{j=2}^D f_j \Biggr) + \Biggl(B G_1 + C \frac{\partial}{\partial t}  \Biggr) \sum_{j=2}^D f_j\nonumber\\
& =  \Biggl( (A-E_1) \frac{\partial }{\partial t} + B G_1 +E_1 A \Biggr) w_1 + \Biggl(B G_1 + C\Bigl(\frac{\partial}{\partial t} +E_1\Bigr)  \Biggr) \sum_{j=2}^D f_j\nonumber \\
&=\Biggl(\lambda^2(p_0+p_1-1) + \lambda(p_0+p_1-2)\frac{\partial}{\partial t}  + \lambda(p_0-1) \frac{\partial}{\partial x_1} - \frac{\partial^2}{\partial t \partial x_1}\Biggr)w_1 \nonumber \\
& \ \ \ + \Biggl(\lambda (p_0+p_1)\Bigl(\frac{\partial}{\partial t} + \lambda\Bigr)+ \lambda p_0 \frac{\partial}{\partial x_1}\Biggr)\sum_{j=2}^D f_j \nonumber\\
& = \Lambda_1 w_1 + \Gamma_1 \sum_{j=2}^D f_j.\label{formulaAB1}
\end{align}

By iterating the above argument, at the $n-th$ step, $n=2,\dots, D$, we have, with $w_{n} = w_{n-1}+f_n$ (meaning that $w_i = \sum_{j=0}^i f_j,\ i = 1,\dots,D$),
\begin{equation}\label{sistemaIterazionen}
\begin{cases}
\displaystyle\frac{\partial^n w_{n-1}}{\partial t^n }  = \Lambda_{n-1} w_{n-1} + \Gamma_{n-1} \sum_{j=n}^D f_j \ \implies\ \Bigl(\frac{\partial^n }{\partial t^n}-\Lambda_{n-1}\Bigr)w_{n-1} = \Gamma_{n-1} f_n + \Gamma_{n-1}  \sum_{j=n+1}^D f_j\\
\displaystyle\Bigl(\frac{\partial }{\partial t} + E_n\Bigr) f_n = G_n w_n + G_n\sum_{j=n+1}^D f_j\\
\displaystyle\Bigl(\frac{\partial }{\partial t} +E_i\Bigr) f_i= G_i w_i + G_i\sum_{j=i+1}^D f_j,\ \ i=n+1,\dots, D.
\end{cases}
\end{equation}
Thus, by using the first two equations of (\ref{sistemaIterazionen}),
\begin{align}
&\Biggl(\frac{\partial^n }{\partial t^n}-\Lambda_{n-1}\Biggr) \Biggl(\frac{\partial }{\partial t} + E_n\Biggr) w_n  \label{formulaIterazionenDerivatan+1} \\
& = \Biggl(\frac{\partial^n }{\partial t^n}-\Lambda_{n-1}\Biggr) G_n w_n + \Gamma_{n-1}\Biggl(\frac{\partial }{\partial t} + E_n\Biggr)f_n +\Biggl[   \Biggl(\frac{\partial }{\partial t} + E_n\Biggr)\Gamma_{n-1} +  \Biggl(\frac{\partial^n }{\partial t^n}-\Lambda_{n-1}\Biggr) G_n\Biggr ]\sum_{j=n+1}^D f_j  \nonumber\\
& = \Biggl(\frac{\partial^n }{\partial t^n}-\Lambda_{n-1} + \Gamma_{n-1}\Biggr) G_n w_n + \Biggl[   \Biggl(\frac{\partial }{\partial t} + E_n\Biggr)\Gamma_{n-1} +  \Biggl(\frac{\partial^n }{\partial t^n}-\Lambda_{n-1} +\Gamma_{n-1}\Biggr) G_n\Biggr ]\sum_{j=n+1}^D f_j. \nonumber
\end{align}
Hence, by reordering the terms in (\ref{formulaIterazionenDerivatan+1}), for $n = 2,\dots, D$, we have that 
$$\Lambda_n = \Biggl(\frac{\partial }{\partial t} + \frac{\partial }{\partial x_n} +\lambda\Biggr)\Lambda_{n-1} + \lambda (p_0-1)\frac{\partial^n }{\partial t^n} + \lambda p_0(\Gamma_{n-1}-\Lambda_{n-1}) - \frac{\partial^{n+1} }{\partial t^{n}\partial x_n}$$
and 
$$\Gamma_n = \Biggl(\frac{\partial }{\partial t} + \frac{\partial }{\partial x_n} +\lambda\Biggr)\Gamma_{n-1} + \lambda p_n\Bigl(\frac{\partial^n }{\partial t^n} +\Gamma_{n-1}- \Lambda_{n-1} \Bigr),$$
with $\Lambda_1, \Gamma_1$ given in (\ref{formulaAB1}).

The interested reader can check (for instance by induction) that the operators $\Lambda_n$ and $\Gamma_n$ are such that
\begin{align}
&\frac{\partial^{n+1} w_n}{\partial t^{n+1}}  = \Lambda_n w_w + \Gamma_n \sum_{j=n+1}^D f_j\label{formulaDerivatawn}\\
& = \Biggl(\,\sum_{k=0}^n \sum_{i\in \mathcal{C}^{\{1,\dots, n\}}_k} \sum_{h = 0}^{n-k} \lambda^{n+1-(h+k)} \biggl[\binom{n-k}{h}\Bigl(p_0 + \sum_{\substack{j = 1\\j\not \in  i}}^n p_j\Bigr) - \binom{n+1-k}{h} \biggr] \frac{\partial ^{h+k}}{\partial t^h \partial x_{i_1}\cdots x_{i_k}}\label{primaSomma} \\
& \ \ \ -\,\sum_{k=1}^{n} \sum_{i\in \mathcal{C}^{\{1,\dots, n\}}_k}  \frac{\partial ^{n+1}}{\partial t^{n+1-k} \partial x_{i_1}\cdots x_{i_k}} \,\Biggr) w_n  \label{secondaSomma}\\
& \ \ \ + \,\sum_{k=0}^n \sum_{i\in \mathcal{C}^{\{1,\dots, n\}}_k} \sum_{h = 0}^{n-k} \lambda^{n+1-(h+k)} \binom{n-k}{h}\Bigl(p_0 + \sum_{\substack{j = 1\\j\not \in  i}}^n p_j\Bigr)  \frac{\partial ^{h+k}}{\partial t^h \partial x_{i_1}\cdots x_{i_k}}  \sum_{j=n+1}^D f_j. \label{terzaSomma}
\end{align}

Finally, for $n=D$ and $w_D = \sum_{j=0}^D f_j = p$, that is the probability density of the position of the motion, formula (\ref{formulaDerivatawn}) reduces to (\ref{relazioneDifferenzialeMotoCompletoMinimale}); indeed, the term in (\ref{terzaSomma}) becomes $0$, the $(D+1)$-th order time derivative can be included in the sum in (\ref{secondaSomma}) as $k=0$ and this new sum becomes the term with $h = D+1-k$ in (\ref{primaSomma}).

%%%%%%%%%%%%%%%%%%%%%%%%%%

%\subsection*{Acknowledgments}

% ---------------------------------------------------------------------------------------------------------



%%       \cite{<label>} ==> Jones et al. [21]
%%       \citep{<label>} ==> [21]
%%


\footnotesize{

}

\end{document}